\documentclass[12pt,a4paper,reqno]{amsart}
\usepackage{amssymb,amsfonts,amsmath,amscd,amsthm,latexsym,euscript}
\usepackage{times}

\usepackage[margin=1in]{geometry}
\usepackage[hidelinks]{hyperref}
\usepackage{tikz}
\usepackage{booktabs}
\usepackage{float}
\usepackage{pdflscape}
\usepackage{array}
\usepackage{multirow}
\usepackage[labelfont=bf]{caption}
\definecolor{up}{HTML}{378ADD}\definecolor{dn}{HTML}{D85A30}
\definecolor{upbg}{HTML}{B5D4F4}\definecolor{dnbg}{HTML}{F5C4B3}
\definecolor{upd}{HTML}{0C447C}\definecolor{dnd}{HTML}{712B13}
\newcommand{\pinK}{0}
\def\cw{3.38}\def\ch{3.05}\def\sx{0.46}\def\sy{0.33}
\def\scw{0.575}\def\srh{0.52}
\newcommand{\walkpanel}[7]{
\begin{scope}[shift={(#1*\cw,-#2*\ch)}]
\ifnum#7=1\def\op{1}\def\tpo{1}\else\def\op{0.22}\def\tpo{0.3}\fi
\ifnum\pinK>0\ifnum#7=1\draw[rounded corners=3pt,black!55,line width=0.7pt] (0.06,-0.05) rectangle (\cw-0.06,-\ch+0.08);\else\draw[rounded corners=3pt,black!14] (0.06,-0.05) rectangle (\cw-0.06,-\ch+0.08);\fi\else\draw[rounded corners=3pt,black!14] (0.06,-0.05) rectangle (\cw-0.06,-\ch+0.08);\fi
\begin{scope}[shift={(0.62,-1.55)}]
\draw[black!55,line width=0.4pt,opacity=\op] (0,0) -- (5.25*\sx,0);
\draw[black!55,line width=0.4pt,opacity=\op] (0,0) -- (0,3.3*\sy);
\foreach \t in {1,...,5} {\draw[black!55,line width=0.4pt,opacity=\op] (\t*\sx,0) -- (\t*\sx,-0.06) node[below,inner sep=1pt,font=\fontsize{5}{6}\selectfont,text=black!60,text opacity=\op] {\t};}
\foreach \h in {0,...,3} {\draw[black!55,line width=0.4pt,opacity=\op] (0,\h*\sy) -- (-0.06,\h*\sy) node[left,inner sep=1pt,font=\fontsize{5}{6}\selectfont,text=black!60,text opacity=\op] {\h};}
\ifnum\pinK>0\draw[black!50,densely dashed,line width=0.4pt,opacity=\op] (\pinK*\sx,0) -- (\pinK*\sx,3.3*\sy);\fi
\draw[#6,line width=1.3pt,line join=round,opacity=\op] (0,0) \foreach \h [count=\i] in {#4} { -- (\i*\sx,\h*\sy) };
\fill[#6,opacity=\op] (0,0) circle (0.9pt);
\foreach \h [count=\i] in {#4} {\fill[#6,opacity=\op] (\i*\sx,\h*\sy) circle (0.9pt);}
\ifnum\pinK>0\ifnum#7=1\fill[black] (\pinK*\sx,0) circle (1.9pt);\fi\fi
\end{scope}
\node[font=\footnotesize,text opacity=\tpo] at (\cw/2,-2.15) {$\sigma=#3$};
\node[font=\footnotesize,text=#6,text opacity=\tpo] at (\cw/2,-2.6) {$#5$};
\end{scope}}
\newcommand{\signrow}[3]{
\node[font=\scriptsize,anchor=east] at (0.32,-#1*\srh) {$p=#2$};
\foreach \s [count=\k] in {#3} {
\if\s+\fill[upbg] (\k*\scw-\scw/2+0.02,-#1*\srh-\srh/2+0.03) rectangle (\k*\scw+\scw/2-0.02,-#1*\srh+\srh/2-0.03);\node[font=\scriptsize,text=upd] at (\k*\scw,-#1*\srh) {$+$};
\else\fill[dnbg] (\k*\scw-\scw/2+0.02,-#1*\srh-\srh/2+0.03) rectangle (\k*\scw+\scw/2-0.02,-#1*\srh+\srh/2-0.03);\node[font=\scriptsize,text=dnd] at (\k*\scw,-#1*\srh) {$-$};\fi}
\pgfmathsetmacro{\kl}{2*#2-1}
\draw[black!60,line width=0.5pt] (\kl*\scw-\scw/2+0.02,-#1*\srh-\srh/2+0.03) rectangle (\kl*\scw+\scw/2-0.02,-#1*\srh+\srh/2-0.03);}

\tikzset{
  walkaxis/.style={->,thin,gray},
  walk/.style={very thick},
  walkdot/.style={circle,fill,inner sep=1.4pt},
  ghost/.style={very thick,densely dashed,gray!80},
  ghostdot/.style={circle,draw,gray!80,fill=white,inner sep=1.3pt},
  steplabel/.style={font=\scriptsize},
  sigmalabel/.style={font=\scriptsize,gray!90!black},
}

\theoremstyle{plain}
\newtheorem{theorem}{Theorem}[section]
\theoremstyle{definition}
\newtheorem{lemma}[theorem]{Lemma}
\newtheorem{proposition}[theorem]{Proposition}
\newtheorem{corollary}[theorem]{Corollary}
\newtheorem{conjecture}[theorem]{Conjecture}
\newtheorem{definition}
   {Definition}
\newtheorem{example}
   {Example}
\newtheorem{remark}
{Remark}

\theoremstyle{definition}

\newtheorem*{claimNo}{Claim}

\newtheorem*{convention}{Convention}

\newtheorem{notation}{Notation}

\theoremstyle{definition}
\theoremstyle{definition}
\theoremstyle{remark}

\newcommand{\btau}{\boldsymbol{\tau}}
\newcommand{\Phip}{\Phi_p}
\newcommand{\inv}{\operatorname{inv}}

\newcommand{\BBR}{\mathbb{R}}
\newcommand{\BBN}{\mathbb{N}} 
\newcommand{\BBS}{\mathbb{S}}
\newcommand{\BBZ}{\mathbb{Z}}

\newcommand{\dd}{\partial}

\DeclareMathOperator{\pt}{pt}

\DeclareMathOperator{\id}{id}
\DeclareMathOperator{\Alt}{Alt}
\DeclareMathOperator{\Wronsk}{Wronsk}

\DeclareMathOperator{\consta}{\boldsymbol{c}}

\newcommand{\by}[1]{\textit{{#1}}}
\newcommand{\jour}[1]{\textit{{#1}}}
\newcommand{\vol}[1]{\textbf{{#1}}}

\title[Even vs 
odd late\/-\/growing permutations${}
\subseteq\mathbb{S}_{2p-1}$:
their counts differ by $(-1)^{p+1}$]{The counts of parity\/-\/odd vs even late\/-\/growing permutations\\[3pt] in the group $\mathbb{S}_{2p-1}$ differ by $(-1)^p$}

\dedicatory{Proof of the claim from arXiv:2605.11137 [math.CO]
}

\author[K.\,C.\,Shah]{Kian C.\ Shah}
\thanks{\textit{Address:} Bernoulli Institute for Mathematics,
  Computer Science and Artificial Intelligence, University of Groningen,
  P.O.\ Box 407, 9700\,AK Groningen, The Netherlands.}

\author[A.\,V.\,Kiselev]{Arthemy V.\ Kiselev${}^{\text{*}}$}
\thanks{${}^{\text{*}}$\:Corresponding author.\quad
  \textit{E-mail:} 
  {\texttt{a.v.kiselev\symbol{"40}rug.nl}}}



\date{26 October 2026}

\begin{document}
\begin{abstract}
A permutation $\sigma\in\BBS_n$ of objects $\{1,\ldots,n\}$ is called \emph{late\/-\/growing} if its partial sums never exceed the respective partial sums of the `average' permutation (with the arithmetic mean $n^{-1}\cdot\sum_{i=1}^n i = (n+1)/2$ of $1,\ldots,n$ on average at each entry $\sigma(i)$): $\sum_{i=1}^k \sigma(i) \leqslant k\cdot (n+1)/2$ for all positive $k\leqslant n$.
We prove that for $n=2p-1$ odd, the set $\Phi_p\subseteq \mathbb{S}_{2p-1}$ of late\/-\/growing permutations is partitioned into parity\/-\/even vs odd elements,
$\Phi_p = \Phi_p^{+} \sqcup \Phi_p^{-}$, in an almost symmetric way:
$| \Phi_p^{+} | - | \Phi_p^{-} | = (-1)^{p+1}$. The proof is constructive:
by viewing the permutations $\sigma\in\Phi_p$ as constrained walks on the integral lattice $\mathbb{Z}$, we describe two filters that pair all but one element in $\Phi_p^{+}$ versus $\Phi_p^{-}$; the remaining walk $\sigma^*$ with monotonically growing step amplitudes determines the dominant parity subset.
\end{abstract}
\maketitle

\tableofcontents
\subsection*{Introduction}
The Wronskian determinants of $2p$ smooth functions appear naturally, up to the integer constants $\consta(p)>0$, from the alternation of weighted differential operators of strict order~$p\in\BBN_{\geqslant 1}$ on the affine real line~$\BBR^1$. 
In this way, the Wronskians generalise the binary Lie brackets of vector fields to $N$-ary Lie structures of higher arities $N=2p$; 
initiated in~\cite{DzhumaWron,DzhumaAssoc,HanlonWachs}
in the context of $L_\infty$-\/algebras (\cite{LadaStasheff,StasheffAinftyLinfty}), 
this line of study continued in~\cite{ForKac,PRG25}, resulting in the verification of Jacobi identities for the strongly homotopy Lie algebra structures given by the (in)complete generalised Wronskian determinants of differential orders${}\geqslant 1$ over the multi\/-\/dimensional base spaces~$\BBR^{d\geqslant 1}$ (see~\cite{ForKac,Yer25}).

In the recent works \cite{Method} and~\cite{cpWron,Valladolid} 
we address the problem of fast calculation of the constants $\consta(p)$, which form the new integer sequence OEIS A392714~\cite{oeisA392714}.
To speed up the accumulation of constants $\consta(p)\sim p^{2p^2}$ obtained by alternating $2p$ differential operators $w_j(x)\cdot \dd_x^p$, 
see Eq.~\eqref{EqMain} 
below, we postulate that the weights are monomials $1$,\ $x$,\ $\ldots$,\ $x^{2p-1}$. This reduces the summation with signs from the entire permutation group $\BBS_{2p}$ to its much smaller subset $\Phi_p\subsetneq \BBS_{2p}$, which we recognise as the set of \emph{late\/-\/growing} permutations (their definition is recalled in what follows, see p.~\pageref{DefLateGrow} in~\S\ref{SecLateGrow}).

In the paper~\cite{Method} immediately preceding this text we observed an interesting property of the sets~$\Phi_p\subseteq\BBS_{2p-1}$ of late\/-\/growing permutations of $2p-1$ objects. For all $p\geqslant 1$, the number $|\Phi_p|$ of elements is odd, and every such set is almost equally partitioned by its subsets of parity\/-\/even and parity\/-\/odd permutations: the difference of these cardinalities equals $|\Phi_p^{-}| - |\Phi_p^{+}| = (-1)^p$.
We now prove this claim by viewing the late\/-\/growing permutations as constrained walks on~the integral lattice~$\BBZ$. In turn, the interpretation of late\/-\/growing permutations 
as walks within $\BBZ_{\geqslant 0}$ reveals a surprising property of the constants $\consta(p)$ that count the Wronskian determinants in~\eqref{EqMain}.

\section{Preliminaries: why late\/-\/growing permutations}\label{SecLateGrow}
\noindent%
Denote by $(-)^\sigma$ the sign of a permutation~$\sigma$.\footnote{\label{FootPermutPhoto}
\textbf{Convention.}\quad We view permutations $\sigma\in\BBS_{n}$ of objects $1,\ldots,n$ as the objects' 
photograph after a permutation acts: for $1\leqslant i \leqslant n$, the value $\sigma(i)$ is the label of object found on the $i$th place. 
For example, $\sigma=(3,2,4,1)\colon (1,x,x^2,x^3) \longmapsto (x^2,x,x^3,1)$.
The sign $(-)^\sigma$ thus equals $-1$ raised to the number of inversions $\sigma(i) > \sigma(j)$ for all $1\leqslant i < j \leqslant n$.}
Let $p\in\BBN_{\geqslant 1}$ and alternate $N=2p$ differential operators $w_j(x)\cdot \dd_x^p$ of strict order~$p$ on the line~$\BBR$ with affine coordinate~$x$;
the result is again an operator of strict differential order~$p$:
\begin{multline}
\Alt\bigl(\bigl\{ w_j(x)\cdot\dd_x^p,\ 1\leqslant j\leqslant N\bigr\}\bigr) =
\sum
_{\sigma\in\BBS_N} (-)^\sigma\, w_{\sigma(1)}\dd_x^p\circ\ldots\circ w_{\sigma(N)}\dd_x^p ={}\\
{}= \consta(p)\cdot\Wronsk\bigl(w_1,\ldots,w_N\bigr)\cdot \dd_x^p,
\label{EqMain}
\end{multline}
where $\Wronsk(w_1,\ldots,w_N) = \det\bigl( \dd_x^{i-1} (w_j) \bigr)$ is the Wronskian determinant of the $N$ coefficients ($1\leqslant i,j \leqslant N$) and $\consta(p)\in\BBZ$ is a universal constant (independent of the weights $\{w_j\}$): e.g., $\consta(1)=1$ and $\consta(2)=2$ yet $\consta(18)=4.881\ldots\cdot 10^{462}$ (see~\cite{cpWron,Method,Valladolid}). 

To calculate the values $\consta(p)>0$ of the new integer sequence OEIS A392\,714 (\cite{oeisA392714}) for $p\geqslant 4$, we take the \emph{monomial} weights $w_j \mathrel{{:}{=}} x^{j-1}$; their Wronskian determinant is $W \mathrel{{:}{=}} \prod_{k=0}^{N-1} k! \in\BBN$, that is a monomial of degree zero.
Whenever the monomials $1$,\ $x$,\ $\ldots$,\ $x^{N-1}$ are chosen as the operator weights, not all permutations $\sigma\in\BBS_{N}$ contribute to the right\/-\/hand side of~\eqref{EqMain}. Indeed, each composition $\dd_x^p\circ w_{\sigma(i)}$ 
--\,preceding the rightmost derivative $\dd_x^p$, i.e.\ within the coefficient $\Wronsk\bigl(\{ w_j \}\bigr)$\,-- multiplies its argument by the weight $w_{\sigma(i)} = x^{\sigma(i)-1}$; then (if $i>1$) the $p$th order derivative $\dd_x^p$ acts from the left. The monomial degree of the term under construction is thus increased by the variable quantity $\sigma(i)-1$ and immediately decreased by the fixed step~$p$.
Yet the monomial degree must 
stay nonnegative at every moment while the operators act from right to left.
This lower bound is the filter\footnote{\label{FootRLHeightLRLateGrow}
The tracking of monomial degrees${}\geqslant 0$ from right to left in the~\eqref{EqBanSet} or~\eqref{eq:filter} is equivalent to the \emph{upper} bounds for the monomial degrees $\sigma(i)-1$ whenever we track them left to right as $i$ grows from $1$ to~$N$. The former set of lower bounds appears in the proofs in this paper; the latter set of upper bounds allows us to recognise the contributing permutations $\sigma\in\BBS_{2p-1}\subsetneq\BBS_{2p}$ as the \emph{late\/-\/growing} permutations, see Definition~\ref{DefLateGrow} on p.~\pageref{DefLateGrow}.}
that selects the permutations which contribute to the r.-h.s.\ of~\eqref{EqMain}.
Specifically, the conditions that restrict the set of actually contributing permutations are:\\[0.5pt]
\mbox{ }\quad (\textit{a})\ %
$\sigma(1)=1$, because $
\deg\Wronsk\bigl(1,x,\ldots,x^{N-1}\bigr) = 0$, yet the leftmost weight $w_{\sigma(1)}$, acting last, is never differentiated;
\\[0.5pt]
\mbox{ }\quad (\textit{b})\ %
for $1\leqslant k < N$ and immediately preceding the multiplication by the weight $w_{\sigma(k)}$, the so far accumulated monomial degree cannot be negative:
\begin{equation}\label{EqBanSet}
\sum_{i=k+1}^N \bigl(\underline{\sigma(i)-1} - p \bigr) \geqslant 0 \qquad
\Longleftrightarrow\qquad
\sum_{i=k+1}^N \sigma(i) \geqslant (N-k)\cdot(p+1).
\end{equation}

\begin{notation}
Let $p\in\BBN_{\geqslant 1}$. We denote by $\Phi_p$,
\begin{equation}
\Phi_p = \bigl\{ \sigma\in\BBS_{2p} \mid \sigma(1)=1 \text{ and inequalities~\eqref{EqBanSet} hold} \bigr\} 
\subseteq \bigl\{ \sigma\in\BBS_{2p} \mid \sigma(1)=1 \bigr\} \simeq \BBS_{2p-1} \subsetneq \BBS_{2p},\label{EqDefPhiP}
\end{equation}
the set of actually contributing permutations (with respect to our choice $w_j \mathrel{{:}{=}} x^{j-1}$ of the weights in the $N=2p$ operators to alternate in~\eqref{EqMain}).\footnote{\label{FootFeelPhiP}
It is readily seen that permutations (of the set $\{1,2,\ldots,N\}$) which are `close to' the identity mapping $\id(i)=i$: $(1,2,\ldots,N)\longmapsto (1,2,\ldots,N)$ can contribute, whereas the permutations which are close to the (almost-) reversion $(1,2,\ldots,N)\longmapsto (1,N,N-1,\ldots,2)$ make no contribution to the r.\/-\/h.s.\ of~\eqref{EqMain} if $p>1$.}
\end{notation}

\begin{table}[htb]
\centering
\caption{For each $p \leqslant 6$, we report the total number of permutations 
$N! = (2p)!$ and $(N-1)!=(2p-1)!$, 
the size of the contributing set $|\Phi_p|$, its fraction of the size of the groups
$\BBS_{2p-1}$ and $\BBS_{2p}$, and the number of parity\/-\/even and odd permutations in $\Phi_p$.
The values $\consta(p \leqslant 18)$ are found in~\cite{Method}.}\label{TabPhiP}
\renewcommand{\arraystretch}{1.25}
\begin{tabular}{c r r r r r r r r}
\hline\hline
$p$ & $|\Phi_p|$ & $(N-1)!$ & $|\Phi_p|/(N-1)!$ & $N!=(2p)!$ & $|\Phi_p|/N!$ & even & odd 
 \\
\hline
$1$ & $1$ & $1$ & $1$ & $2$ & $1/2$ & $1$ & $0$ 
\\
$2$ & $3$ & $6$ & $1/2$ & $24$ & $1/8$ & $1$ & $2$ 
\\
$3$ & $35$ & $120$ & $\approx 1/3$ & $720$ & $\approx1/20$ & $18$ & $17$ 
\\
$4$ & $1\,001$ & $5\,040$ & $\approx 1/5$ & $40\,320$ & $\approx1/40$ & $500$ & $501$ 
\\
$5$ & $53\,109$ & $362\,880$ & $\approx 1/7$ & $3\,628\,800$ & $\approx1/68$ & $26\,555$ & $26\,554$ 
\\
$6$ & $4\,605\,271$ & $39\,916\,800$ & $\approx 1/9$ & $479\,001\,600$ & $\approx1/104$ & $2\,302\,635$ & $2\,302\,636$ 
\\
\hline\hline
\end{tabular}
\end{table}

We observe from Table~\ref{TabPhiP} that as $p$ grows, the set $\Phi_p$ constitutes a decreasing (in proportion to the size of the 
permutation group) part of $\BBS_{2p-1}\subsetneq\BBS_{2p}$; here we recall that $\sigma(1)\equiv 1$ and only the objects $2,\ldots,N$ are permuted by $\sigma\in\Phi_p\subsetneq\BBS_{2p}$.
In fact, it is the \emph{preceding} number of objects, $2p-2$, and the permutation group $\BBS_{2p-2}$ of size $(2p-2)!$ which, in the limit $p\to+\infty$, captures all the information about the late\/-\/growing permutations of $2p-1$ objects $2$,\ $\ldots$,\ $2p$ (see App.~\ref{AppLateGrowLimRatios} on p.~\pageref{AppLateGrowLimRatios}).

\begin{claimNo}[\cite{Method}]
As $p\to+\infty$, the ratio $|\Phi_p| / (2p-2)! $ tends monotonically from above to~$1$.\footnote{\label{FootSingapore}
Wenlin Zhang (from the National University of Singapore) communicated to us an AI-\/generated sketch of the proof; the excess $|\Phi_p| / (2p-2)! - 1$ tends to $+0$ as slow as $O\bigl(\log(p)/p\bigr)$.
Let us emphasize that the counts of late\/-\/growing permutations for $n=2p-1$ objects (as in this paper) and for $n=2\ell$ objects behave visibly differently. Both the limits (pro ratio to $(2p-2)!$ and $(2\ell-1)!$, respectively) equal unit as $p\to+\infty$ and $\ell\to+\infty$, yet the excess over $1$ for the even indices $n=2\ell \gg 1$ is less than the excess for the neighbouring odd indices $n=2\ell\pm1$ of the integer sequence A147681~(\cite{oeisA147681}).
If the excess decay law $O\bigl(\log(p)/p\bigr)$ is common for these two subsequences of A147681$(n)$, the constants in it are different for odd \textit{vs} even indices~$n$.}
\\%
\centerline{\rule{1in}{0.7pt}}
\end{claimNo}

\smallskip
\noindent%
Recalling that $1+2+\ldots+N=\sum_{i=1}^N\sigma(i) = N(N+1)/2$ and that $\sigma(1)=1$ by the above condition~(\textit{a}), from~\eqref{EqBanSet} in condition~(\textit{b}) we deduce that
\[
\sum_{i=2}^k \sigma(i) \leqslant \tfrac{1}{2}N(N+1) - 1 - (N-k)\cdot(p+1) \qquad
\text{for all $k$ such that 
$2\leqslant k \leqslant N$.}
\]
Noting that the object labelled `1' is never moved by the contributing permutations $\sigma$ from~\eqref{EqDefPhiP},
let us permute not the succeeding objects $\{ 2,\ldots, N\}$ but, by putting $\sigma(i) = 1 + \varsigma(i')$ for $i'=i-1$, the relabelled, smaller set $\{\mathsf{1},\ldots,\mathsf{N-1} \}$ using permutations $\varsigma\in\BBS_{2p-1}$. Keeping in mind that $p=N/2$, we balance
\[
k-1 + \sum_{i'=1}^{k-1} \varsigma(i') \leqslant \tfrac{1}{2}N(N+1) - 1 - \bigl( N - [k-1] - 1 \bigr) \cdot \bigl( \tfrac{1}{2}N + 1 \bigr).
\]
By simple arithmetic we conclude that
\[
\sum_{i'=1}^{k-1} \varsigma(i') \leqslant (k-1)\cdot \frac{N}{2} = (k-1)\cdot\frac{\bigl(\text{\#\,objects${}+1$}\bigr)}{2}  \qquad \text{for $\varsigma\in\BBS_{N-1}$, $1\leqslant k-1 \leqslant N-1$.}
\]
Thus, the filter in~\eqref{EqBanSet} that worked backwards, i.e.\ from right to left along permutations $\sigma$ in the set $\Phi_p$, see~\eqref{EqDefPhiP}, is equivalent to the above filter that works progressively, from left to right along the permutations~$\varsigma$. It is this second phrasing of the filter which gives the name to those permutations which pass through. 

\begin{definition}\label{DefLateGrow}
A permutation $\varsigma\in\BBS_n$ is \emph{late\/-\/growing} if its partial sums never exceed the corresponding partial sums of the `average' permutation (with the arithmetic mean of $1,\ldots,n$ at each position on average over the group~$\BBS_n$),
that is,
\[
\sum_{i=1}^\kappa \varsigma(i) \leqslant \kappa\cdot \frac{n+1}{2} \qquad \text{for all $\kappa = 1,\ldots, 
n$.}
\]
Indeed, a permutation of $\{1,\ldots,n\}$ has $n^{-1}\cdot n(n+1)/2$ on average at each of its $n$ positions.
\end{definition}

\begin{remark}\label{RemParityVsShift}
The parity $(-)^\varsigma$ of a late\/-\/growing permutation $\varsigma$ equals the parity of the respective permutation $\sigma\in\Phi_p$ that does not transpose the object `1' with any other object `2',\ $\ldots$,~`$N$'.\footnote{
Indeed, neither by shifting the labels by $+1$ nor by adding the new object `1', with its label preceding all the other labels `2',\ $\ldots$,\ `$N$' after the shift, are any pre\/-\/existing inversions destroyed or any previously non\/-\/existent inversions created.
The parity $(-)^\varsigma$ of a permutation $\varsigma\in\Phi_p\subsetneq\BBS_{2p-1}$ thus equals the parity $(-)^\sigma$ of the permutation $\sigma\in\BBS_{2p}$ in the description of the set~$\Phi_p$ by using the postulate $\sigma(1)=1$ and assembly $\sigma(i) = 1 + \varsigma(i-1)$ for $2\leqslant i \leqslant N=2p$.}
\end{remark}

In this work we prove that the set $\Phi_p\subsetneq \BBS_{2p-1}$ of late\/-\/growing permutations (of $\{\mathsf{1}$,\ $\ldots$,\ $\mathsf{2p-1}\}$) is almost equally partitioned by parity\/-\/even and parity\/-\/odd elements, the difference of their numbers interlacing as~$(-1)^{p+1}$. 
Let us emphasize that it is the realisation of the set $\Phi_p\subseteq \BBS_{2p-1}$ by using the objects $\{1,\ldots,N=2p\}$ and their permutations $\sigma\in\BBS_{2p}$ satisfying the conditions (\textit{a})--(\textit{b}), see above, 
which we use in the proof.
Our reasoning is combinatorial, yet
the alternation of $N=2p$ operators with monomial weights in Eq.~\eqref{EqMain} serves the context: by looking at the monomial degrees, we view a contributing permutation $\sigma\in\Phi_p$ as a constrained walk along $\BBZ_{\geqslant 0}$ in
the lattice~$\BBZ$; the ban for monomial degrees to pass below zero selects the permutations for~$\Phi_p$ by the set of inequalities~\eqref{EqBanSet} with $1\leqslant k < N$. 

\section{The filter $\sigma\in\Phi_p$ as a walk}\label{sec:setup}
\noindent%
By definition, put\footnote{\label{FootNotationFromGen2}
The notation agrees with~\cite{Method}: this is the running sum which the permutation generator (ver.~2 in~\cite{Method}, contrasted with ver.~1 in~\cite{cpWron,Valladolid}) tracks while it fills the positions $\sigma(N)$,\ $\sigma(N-1)$,\ $\ldots$
from right to left, 
adding $R-(p+1)$ for each new entry~$R \mathrel{{=}{:}} \sigma(i)$.}
\begin{equation}\label{eq:filter}
  T_k \mathrel{{:}{=}} \sum_{i=N-k+1}^{N}\bigl(\sigma(i)-1-p\bigr) \geqslant 0
  \qquad\text{for } k=1,\dots,N-1 .
\end{equation}
We denote the \emph{increments} by 
\[
  \tau_j \;:=\; \sigma(N+1-j)-(p+1), \qquad j=1,\dots,N-1,
\]
so that $\btau=(\tau_1,\dots,\tau_{N-1})$ is the permutation $\sigma$ read from the right with its first
entry $\sigma(1)=1$ dropped and every entry shifted down by $p+1$, and
$T_k=\tau_1+\cdots+\tau_k$. Since $\sigma(2)$,\ $\dots$,\ $\sigma(N)$ are the numbers $2$,\ $\dots$,\ $2p$
in some order, the tuple $\btau$ is an ordering of the $2p-1$ numbers
\begin{equation}\label{EqCenteredSet}
  V_p \;=\; \{-(p-1),\,\dots,\,-1,\,0,\,1,\,\dots,\,p-1\}.
\end{equation}
Conversely, every ordering of the set $V_p$ arises from exactly one permutation $\sigma$ with $\sigma(1)=1$.

\begin{definition}
The running sum $T_k$ is the \emph{height} after $k$ steps;
the ordered $(N-1)$-tuple of increments $\btau$ is \emph{legal} if all its heights are $\geqslant 0$. 
\end{definition}

In these words,
\[
  \Phip \;=\; \{\sigma\in \BBS_{2p} : \sigma(1)=1 \text{ and } \btau \text{ is legal}\},
\]
and $T_{N-1}=\sum_{v\in V_p}v=0$: a legal tuple $\btau$ is a walk that starts at height $T_0=0$, never goes below $0$, and ends at height~$T_{N-1}=0$.

\begin{convention}[on the reading direction]
In what follows, everything 
is phrased in terms of the increments $\btau$, that is, in terms of the permutation $\sigma$ \emph{read from right to left}. 
The $j$th slot of the increment tuple $\btau$ is position $N+1-j$ of $\sigma$: 
the slot $1$ is 
the last entry $\sigma(N)$, yet the slot $N-1$ is 
$\sigma(2)$. 
Whenever we say that something happens
\emph{before} or \emph{after} something else, or if we speak of the \emph{first} slot, or of
\emph{the height before a pair of slots}, we mean the order of $\btau$; in the permutation $\sigma$ this is the order from right to left. 
So ``the height before slots $2i,2i+1$'' is
$T_{2i-1}=\tau_1+\cdots+\tau_{2i-1}$, the sum of the $2i-1$ \emph{rightmost} entries of
$\sigma$ after the shift. Of course, this is also the order 
in which the operators act in Eq.~\eqref{EqMain}:  
the increment $\tau_1$ comes from the innermost weight $w_{\sigma(N)}=x^{\sigma(N)-1}$, which is differentiated first, and the running sum $T_k$ is exactly the exponent of $x$ that is left after the
$k$-th block $\partial_x^{p}$ has acted. ``Legal'' therefore means: the exponent of $x$ never becomes negative, which is the filter in~\eqref{EqBanSet}.
\end{convention}

\begin{example}\label{ex:reading}
Let $p=3$, so $N=6$, and take the permutation $\sigma=(1,2,5,3,6,4)$. Reading it from the right and subtracting $p+1=4$, we obtain:
\[
  \sigma=(\,1,\ \underbrace{2}_{\tau_5=-2},\ \underbrace{5}_{\tau_4=1},\
  \underbrace{3}_{\tau_3=-1},\ \underbrace{6}_{\tau_2=2},\ \underbrace{4}_{\tau_1=0}\,)
  \qquad\longleftarrow\ \text{read this way,}
\]
so the increments are $\btau=(0,2,-1,1,-2)$ with heights $\boldsymbol{T}=(0,2,1,2,0)$, all $\geqslant 0$, hence legal, so $\sigma\in\Phi_3$. 
In terms of the alternated differential operators of order $p=3$ with monomial coefficients, the innermost weight is $w_4=x^{3}$, and
$\partial_x^{3} (x^{3})$ leaves $x^{0}$ (height $0$);\footnote{\label{FootOnlyDegrees}
In our present study of late\/-\/growing permutations, viewed as constrained walks along the lattice~$\BBZ$, it is only the integral monomial degrees which we keep track~of. Yet in~\cite{Method} and~\cite{Valladolid} we gradually accumulate the constants $\consta(p)$ from the products of falling factorials $E^{\underline{p}}_k$ that appear from the $p$th order derivatives of monomials at each step $k=N$,\ $N-1$,\ $\ldots$,~$1$.}
multiplying by $w_6=x^{5}$ and
differentiating three times leaves $x^{2}$ (height $2$); then $w_3=x^{2}$ takes it to
$x^{4}$ and $\partial_x^{3}$ leaves $x^{1}$; then $w_5=x^{4}$ gives $x^{5}\to x^{2}$; then
$w_2=x$ gives $x^{3}\to x^{0}$. The heights $(0$,\ $2$,\ $1$,\ $2$,\ $0)$ are these exponents, and the
values $E_k = T_k + p$ introduced in~\cite{cpWron,Method,Valladolid}
are the exponents before each differentiation, $E_k = (3,5,4,5,3) = T_k+3$.

\smallskip
\noindent$\bullet$\quad 
Now take the permutation $\sigma=(1,3,6,4,2,5)$. 
Reading it from the right and subtracting $p+1=4$ from each but the leftmost entry $\sigma(1)=1$, read last, gives us the ordered increments $\btau=(1,-2,0,2,-1)$ and
heights $T_1=1$ and $T_2=-1<0$: the walk goes below zero at time $2$, so
$\sigma\notin\Phi_3$. 
In terms of the differential operators, 
the weight $w_5=x^{4}$ is differentiated to $x^{1}$,
then multiplied by $w_2=x$ to give $x^{2}$, and $\partial_x^{3}(x^{2})=0$, so this term dies.
Both walks are drawn in Figure~\ref{fig:walks}.
\end{example}

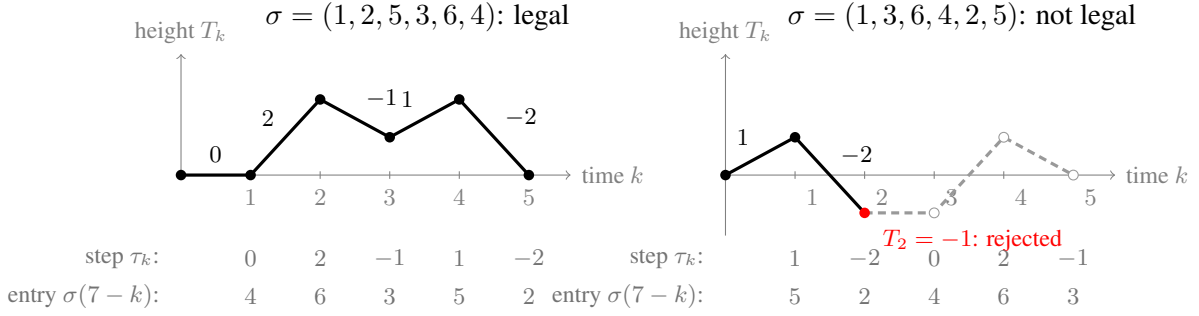
\begin{figure}[ht]
\centering
\begin{tikzpicture}[x=0.92cm,y=0.5cm]
  \begin{scope}
    \draw[walkaxis] (0,0) -- (5.6,0) node[right,font=\scriptsize] {time $k$};
    \draw[walkaxis] (0,0) -- (0,3.2) node[above,font=\scriptsize] {height $T_k$};
    \foreach \k in {1,...,5} {\draw[gray,thin] (\k,0.1) -- (\k,-0.1) node[below,font=\scriptsize] {$\k$};}
    \draw[walk] (0,0) -- (1,0) -- (2,2) -- (3,1) -- (4,2) -- (5,0);
    \foreach \pt in {(0,0),(1,0),(2,2),(3,1),(4,2),(5,0)} {\node[walkdot] at \pt {};}
    \node[steplabel,above] at (0.5,0.05) {$0$};
    \node[steplabel,above left] at (1.5,1) {$2$};
    \node[steplabel,above right] at (2.5,1.5) {$-1$};
    \node[steplabel,above left] at (3.5,1.5) {$1$};
    \node[steplabel,above right] at (4.5,1) {$-2$};
    \node[sigmalabel,anchor=east] at (-0.15,-2.2) {step $\tau_k$:};
    \foreach \k/\s in {1/0,2/2,3/-1,4/1,5/-2} {\node[sigmalabel] at (\k,-2.2) {$\s$};}
    \node[sigmalabel,anchor=east] at (-0.15,-3.2) {entry $\sigma(7-k)$:};
    \foreach \k/\s in {1/4,2/6,3/3,4/5,5/2} {\node[sigmalabel] at (\k,-3.2) {$\s$};}
    \node[font=\small] at (3.4,4.2) {$\sigma=(1,2,5,3,6,4)$: legal};
  \end{scope}
  \begin{scope}[xshift=7.2cm]
    \draw[walkaxis] (0,0) -- (5.6,0) node[right,font=\scriptsize] {time $k$};
    \draw[walkaxis] (0,-1.6) -- (0,3.2) node[above,font=\scriptsize] {height $T_k$};
    \foreach \k in {1,...,5} {\draw[gray,thin] (\k,0.1) -- (\k,-0.1) node[below right,font=\scriptsize] {$\k$};}
    \draw[walk] (0,0) -- (1,1) -- (2,-1);
    \draw[ghost] (2,-1) -- (3,-1) -- (4,1) -- (5,0);
    \foreach \pt in {(0,0),(1,1)} {\node[walkdot] at \pt {};}
    \node[walkdot,red] at (2,-1) {};
    \foreach \pt in {(3,-1),(4,1),(5,0)} {\node[ghostdot] at \pt {};}
    \node[steplabel,above left] at (0.5,0.5) {$1$};
    \node[steplabel,above right] at (1.5,0) {$-2$};
    \node[font=\scriptsize,red,anchor=north west] at (2.1,-1.15) {$T_2=-1$: rejected};
    \node[sigmalabel,anchor=east] at (-0.15,-2.2) {step $\tau_k$:};
    \foreach \k/\s in {1/1,2/-2,3/0,4/2,5/-1} {\node[sigmalabel] at (\k,-2.2) {$\s$};}
    \node[sigmalabel,anchor=east] at (-0.15,-3.2) {entry $\sigma(7-k)$:};
    \foreach \k/\s in {1/5,2/2,3/4,4/6,5/3} {\node[sigmalabel] at (\k,-3.2) {$\s$};}
    \node[font=\small] at (3.4,4.2) {$\sigma=(1,3,6,4,2,5)$: not legal};
  \end{scope}
\end{tikzpicture}
\caption{The two permutations of Example~\ref{ex:reading} viewed as walks (here $p=3$ and $N=6$). The `time' runs
from right to left in permutations~$\sigma$: the entry used at time~$k$ is $\sigma(N+1-k)=\sigma(7-k)$,
and the step taken at time $k$ is $\tau_k=\sigma(7-k)-4$. The height after $k$ steps is
the exponent of~$x$ left after the $k$th differentiation~$\dd_x^p$. 
On the right diagram, the walk goes below zero at time~$2$, so the permutation is rejected by the filter; the dashed part is never reached.}
\label{fig:walks}
\end{figure}

\begin{definition}\label{DefStaircase}
Put $m \mathrel{{:}{=}} p-1$.
The \emph{staircase} is the increments' ordering 
\[
\btau^{*}=(0,1,-1,2,-2,\dots,m,-m)
\]
in which the \emph{payments} are $x_j=j > 0$, immediately followed by 
the \emph{debts} $y_j=-j$  for all $j$ (cf.\ Proposition~\ref{prop:shape} on p.~\pageref{prop:shape} below).
These are the increments of the permutation
\[
  \sigma^{*}=(1,\,2,\,2p,\,3,\,2p-1,\,4,\,2p-2,\,\dots,\,p,\,p+2,\,p+1).
\]
\end{definition}


\begin{theorem}\label{thm:parity}
For every $p\in\BBN_{\geqslant 1} $,
\[
  \sum_{\sigma\in\Phip} (-)^{\sigma}\;=\; (-)^{\sigma^{*}}\;=\; (-1)^{p+1} .
\]
In consequence, the counts of
parity\/-\/even and parity\/-\/odd permutations in $\Phi_p$ differ 
by exactly one: the even permutations being ahead for odd $p$ and the odd permutations ahead for even~$p$.
For the same reason, the number of elements $|\Phip|$ is odd for every~$p$.
\end{theorem}

The rest of the paper contains the proof of Theorem~\ref{thm:parity}, accompanied with examples which illustrate every new notion and every step of our reasoning.
The idea of the proof is that we introduce two moves (of elements in the ordered increment tuples $\btau$, hence in the permutations~$\sigma$)
that pair up permutations of opposite parities which cancel each other out, leaving only the permutation \(\sigma^*\) with the staircase increments~$\btau^*$. 
The first move pairs off almost everything; what is left has
a `rigid' shape (see Definition~\ref{DefRigid} below and Eq.~\eqref{EqFormRigid});
the second move pairs off all of the rigid permutations 
except the staircase permutation~$\sigma^{*}$. 
Either move exchanges the entries in two slots of~$\btau$.\footnote{
A move that exchanges two neighbouring entries of a legal increment ordering 
$\btau$ changes exactly one height, so it is easy to see when the result is still legal.} 
Exchanging $\tau_a$ and $\tau_b$ is the same as exchanging $\sigma(N+1-a)$ and $\sigma(N+1-b)$, a transposition
of $\sigma$, so every such move reverses the parity $(-)^{\sigma}$. 
The sign $(-)^{\sigma^*}$ of the staircase permutation is calculated separately, in~\S\ref{sec:staircase} on p.~\pageref{sec:staircase}.

\begin{example}
For $p=4$, we know $|\Phi_4|=1001$, see Table~\ref{TabPhiP} on p.~\pageref{TabPhiP}.
Move~1 pairs off $986$ permutations, 
leaving $15$ not yet processed; then, Move~2 pairs off $14$ of those, leaving the staircase permutation $\sigma^{*}=(1,2,8,3,7,4,6,5)$, which is parity\/-\/odd. 
Thus $500-501=-1$, the sign of~$\sigma^*$ at~hand.

Figure~\ref{fig:venn} charts 
the set $\Phi_4$ with late\/-\/growing permutations and their pairings.
\end{example} 

\begin{figure}[htb]
\centering
\definecolor{oeEven}{RGB}{38,92,170}
\definecolor{oeOdd}{RGB}{186,44,72}
\definecolor{oeMoveA}{RGB}{18,128,112}
\definecolor{oeMoveB}{RGB}{226,132,12}
\definecolor{oeStar}{RGB}{238,178,0}
\begin{tikzpicture}[x=1cm,y=1cm,
  evdot/.style={circle,fill=oeEven,draw=white,line width=0.4pt,inner sep=1.6pt},
  oddot/.style={circle,fill=oeOdd,draw=white,line width=0.4pt,inner sep=1.6pt},
  moveA/.style={oeMoveA,line width=0.9pt},
  moveB/.style={oeMoveB,line width=1.3pt}]
  \begin{scope}
    \clip (0,0) ellipse (6.4 and 3.75);
    \fill[oeEven!10] (-6.5,-3.8) rectangle (0,3.8);
    \fill[oeOdd!10] (0,-3.8) rectangle (6.5,3.8);
  \end{scope}
  \begin{scope}
    \clip (0,0) ellipse (3.25 and 1.85);
    \fill[oeEven!22] (-3.3,-1.9) rectangle (0,1.9);
    \fill[oeOdd!22] (0,-1.9) rectangle (3.3,1.9);
  \end{scope}
  \draw[gray!70,line width=0.6pt] (0,3.75) -- (0,-3.75);
  \draw[black!75,line width=1.1pt] (0,0) ellipse (6.4 and 3.75);
  \draw[black!60,line width=0.9pt,dashed] (0,0) ellipse (3.25 and 1.85);
  \foreach \x in {0.9,1.85,2.8,3.75}{
    \draw[moveA] (-\x,2.15) to[bend left=26] (\x,2.15);
    \draw[moveA] (-\x,-2.15) to[bend right=26] (\x,-2.15);
    \node[evdot] at (-\x,2.15) {}; \node[oddot] at (\x,2.15) {};
    \node[evdot] at (-\x,-2.15) {}; \node[oddot] at (\x,-2.15) {};
  }
  \foreach \x in {0.55,1.25,1.95,2.65}{
    \draw[moveB] (-\x,0.35) to[bend left=30] (\x,0.35);
    \node[evdot] at (-\x,0.35) {}; \node[oddot] at (\x,0.35) {};
  }
  \foreach \x in {0.55,1.25,1.95}{
    \draw[moveB] (-\x,-0.35) to[bend right=30] (\x,-0.35);
    \node[evdot] at (-\x,-0.35) {}; \node[oddot] at (\x,-0.35) {};
  }
  \draw[gray!70,densely dashed,line width=0.6pt] (-2.65,-0.35) circle (0.16);
  \fill[oeStar,draw=black!80,line width=0.6pt,shift={(2.65,-0.35)}]
    (90:0.36) -- (126:0.15) -- (162:0.36) -- (198:0.15) -- (234:0.36) -- (270:0.15)
    -- (306:0.36) -- (342:0.15) -- (18:0.36) -- (54:0.15) -- cycle;
  \foreach \sx in {-1,1}{
    \node[gray!75!black] at (\sx*4.5,2.15) {$\cdots$};
    \node[gray!75!black] at (\sx*4.5,-2.15) {$\cdots$};
  }
  \draw[black!70,line width=0.5pt] (2.95,-0.55) -- (3.75,-1.1);
  \node[font=\small,anchor=west,fill=oeOdd!10,inner sep=1.5pt] at (3.75,-1.1) {$\sigma^{*}$ (staircase)};
  \draw[gray!70,line width=0.5pt] (-2.8,-0.47) -- (-3.75,-1.1);
  \node[font=\small,anchor=east,text=gray!75!black,fill=oeEven!10,inner sep=1.5pt] at (-3.75,-1.1) {no partner};
  \node[align=center,font=\small,text=oeEven!80!black] at (-5.0,0)
    {parity\/-\/even\\[1pt] $\Phi_4^{+}$\\[1pt] $500$};
  \node[align=center,font=\small,text=oeOdd!80!black] at (5.0,0)
    {parity\/-\/odd\\[1pt] $\Phi_4^{-}$\\[1pt] $501$};
  \node[font=\large] at (-5.35,3.45) {$\Phi_4$};
  \node[font=\scriptsize,black!70,fill=white,rounded corners=1pt,inner sep=1.5pt] at (0,-1.4) {rigid};
  \node[font=\scriptsize,black!70,fill=white,rounded corners=1pt,inner sep=1.5pt] at (0,-3.75) {some free pair};
  \begin{scope}[yshift=-4.55cm]
    \draw[moveA] (-6.1,0) to[bend left=26] (-5.4,0);
    \node[font=\footnotesize,anchor=west] at (-5.25,0) {Move~1: $493$ pairs, a few drawn};
    \draw[moveB] (0.75,0) to[bend left=30] (1.45,0);
    \node[font=\footnotesize,anchor=west] at (1.6,0) {Only $7$ pairs remain for Move~2};
  \end{scope}
\end{tikzpicture}
\caption{The set $\Phi_4$ of $1\,001$ late\/-\/growing permutations, split by parity (left: even, right: odd) and by the two moves. Every arc joins an even and an odd permutation. Only the staircase permutation $\sigma^{*}$ is left unpaired.}
\label{fig:venn}
\end{figure}
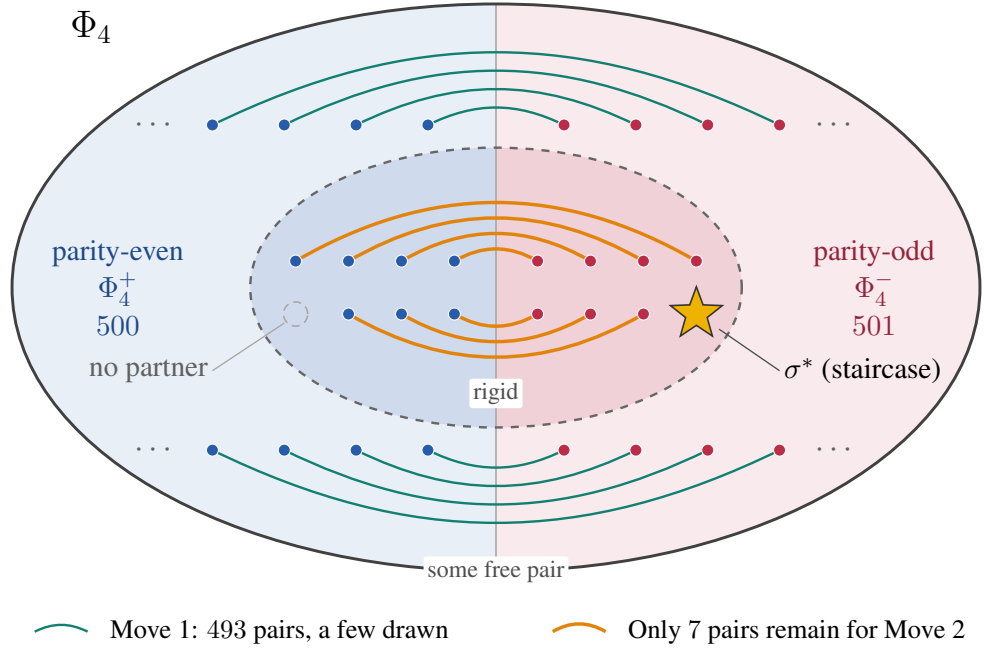

We now describe the two moves of increments for a permutation~$\sigma\in\Phi_p$.

\section{Move 1: exchange inside a pair of neighbouring slots}\label{sec:move1}
\noindent%
Suppose $p>1$. Let us group the increment slots indexed by $2,3,\dots,N-1$ in a tuple $\btau$ into 
the $m \mathrel{{:}{=}} p-1>0$ pairs, themselves indexed by 
\[
  (2,3),\ (4,5),\ \dots,\ (2m,2m+1),
\]
and leave slot $1$ intact. 
Exchanging the two entries of the $i$th pair changes exactly one
height, namely $T_{2i}$, from $T_{2i-1}+\tau_{2i}$ to $T_{2i-1}+\tau_{2i+1}$; every other $T_k$ either stops before the pair $(k \leqslant 2i-1)$, so it does not involve the pair at all, or runs past it $(k \geqslant 2i+1)$, so it involves the pair only through $\tau_{2i} + \tau_{2i+1}$, which does not depend on the two summands' order.
Consequently, for a legal tuple $\btau$ the exchanged ordering is legal precisely when
\begin{equation}\label{eq:free}
  T_{2i-1}+\tau_{2i+1}\;\geqslant\;0.
\end{equation}

\begin{definition}
We say that a pair of increments is \emph{free} when condition \eqref{eq:free} holds. 
\end{definition}

A pair is free when the height before it (that is $T_{2i-1}$, the running sum over the slots to the left of the pair in $\btau$, which are the entries to the right of it in $\sigma$), plus the \emph{second} entry of the pair, is still $\geqslant 0$, so that the second entry could just as well have been taken first.

\begin{example}[$p=4$, so $N=8$]\label{ex:move1}
The permutation $\sigma=(1,2,8,3,5,4,6,7)\in\Phi_4$ yields, reading from the right,
the increments $\btau=(2,\,1,-1,\,0,-2,\,3,-3)$; 
the heights are $\boldsymbol{T}=(2$,\ $3$,\ $2$,\ $2$,\ $0$,\ $3$,\ $0)$. Let us mark the pairs,
\[
  \btau=(\,2,\ \overbrace{1,\,-1}^{\text{pair }1},\ \overbrace{0,\,-2}^{\text{pair }2},\
  \overbrace{3,\,-3}^{\text{pair }3}\,),
  \qquad
  \text{heights before the pairs: } T_1=2,\ T_3=2,\ T_5=0 .
\]
Now let us consider the pairs in their consecutive order.\\[0.5pt]
\mbox{ }\quad Pair~$1$: $T_1+\tau_3=2-1=1\geqslant 0$, free.\\[0.5pt]
\mbox{ }\quad Pair~$2$: $T_3+\tau_5=2-2=0\geqslant 0$, free.\\[0.5pt]
\mbox{ }\quad Pair~$3$: $T_5+\tau_7=0-3 < 0$, not free: the walk stands at height $0$ before it and cannot afford the step $-3$ first. 

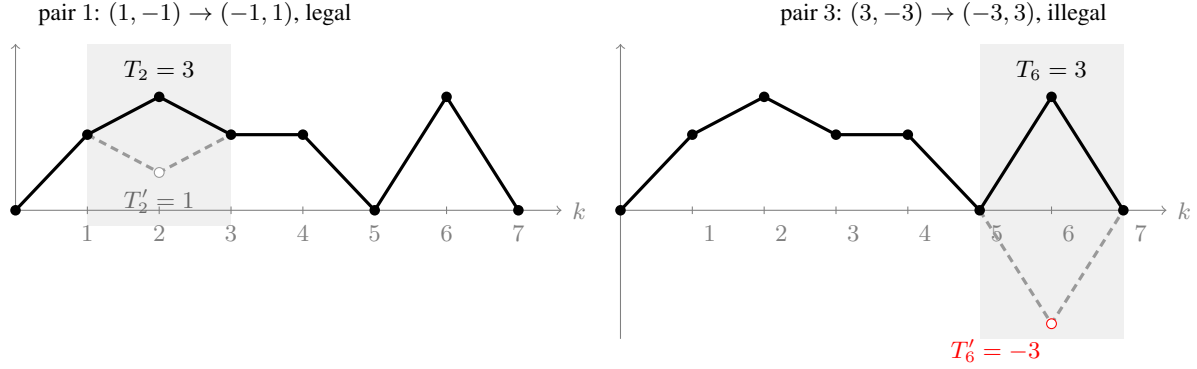
\begin{figure}[htb]
\centering
\begin{tikzpicture}[x=0.95cm,y=0.5cm]
  \begin{scope}
    \fill[gray!12] (1,-0.4) rectangle (3,4.4);
    \draw[walkaxis] (0,0) -- (7.6,0) node[right,font=\scriptsize] {$k$};
    \draw[walkaxis] (0,0) -- (0,4.4);
    \foreach \k in {1,...,7} {\draw[gray,thin] (\k,0.1) -- (\k,-0.1) node[below,font=\scriptsize] {$\k$};}
    \draw[ghost] (1,2) -- (2,1) -- (3,2);
    \node[ghostdot] at (2,1) {};
    \draw[walk] (0,0) -- (1,2) -- (2,3) -- (3,2) -- (4,2) -- (5,0) -- (6,3) -- (7,0);
    \foreach \pt in {(0,0),(1,2),(2,3),(3,2),(4,2),(5,0),(6,3),(7,0)} {\node[walkdot] at \pt {};}
    \node[steplabel,above] at (2,3.15) {$T_2=3$};
    \node[steplabel,below,gray!80!black] at (2,0.85) {$T_2'=1$};
    \node[font=\scriptsize] at (2.5,5.2) {pair $1$: $(1,-1)\to(-1,1)$, legal};
  \end{scope}
  \begin{scope}[xshift=8.0cm]
    \fill[gray!12] (5,-3.4) rectangle (7,4.4);
    \draw[walkaxis] (0,0) -- (7.6,0) node[right,font=\scriptsize] {$k$};
    \draw[walkaxis] (0,-3.4) -- (0,4.4);
    \foreach \k in {1,...,7} {\draw[gray,thin] (\k,0.1) -- (\k,-0.1) node[below right,font=\scriptsize] {$\k$};}
    \draw[ghost] (5,0) -- (6,-3) -- (7,0);
    \node[ghostdot,draw=red] at (6,-3) {};
    \draw[walk] (0,0) -- (1,2) -- (2,3) -- (3,2) -- (4,2) -- (5,0) -- (6,3) -- (7,0);
    \foreach \pt in {(0,0),(1,2),(2,3),(3,2),(4,2),(5,0),(6,3),(7,0)} {\node[walkdot] at \pt {};}
    \node[steplabel,above] at (6,3.15) {$T_6=3$};
    \node[font=\scriptsize,red,below left] at (6,-3.1) {$T_6'=-3$};
    \node[font=\scriptsize] at (4.5,5.2) {pair $3$: $(3,-3)\to(-3,3)$, illegal};
  \end{scope}
\end{tikzpicture}
\caption{In Example~\ref{ex:move1}, the increments are $\btau=(2,1,-1,0,-2,3,-3)$.
Exchanging a pair changes one height only (dashed). 
Left diagram: pair $1$ is free, the exchanged walk is legal.
Right diagram: pair $3$ is not free; exchanging it would send the walk to $-3$ at time~$6$. (Illegal permutations are rejected at once by the generator of the set $\Phi_p$ in~\cite{Method}.)}
\label{fig:move1}
\end{figure}

Exchanging pair~$1$ gives the increments $\btau'=(2,\,-1,1,\,0,-2,\,3,-3)$ 
with heights $(2$,\ $1$,\ $2$,\ $2$,\ $0$,\ $3$,\ $0)$: only $T_2$ changed, from $3$ to $1$. In the permutation $\sigma$ this is the exchange of $\sigma(7)=6$ and $\sigma(6)=4$, giving the new permutation $\sigma'=(1,2,8,3,5,6,4,7)$. Figure~\ref{fig:move1}
shows both walks, together with the illegal result of exchanging pair~$3$.
\end{example}

Notice that the initial assumption $T_{2i-1} + \tau_{2i} \geqslant 0$ and the 
permission of its alternative by condition~\eqref{eq:free} involve the height before
the pair and the two entries of the pair, yet none of these change when any pair 
(this one or another) is exchanged. This proves a useful statement.
\begin{lemma}\label{LemRemainFree}
The set of \emph{free} pairs 
is the same for an increment tuple $\btau$ and for every ordering of increments obtained from~$\btau$ 
by exchanging neighbouring slots within any free pair(s). 
\end{lemma}

\begin{example}
Consider Example~\ref{ex:move1} again. After exchanging pair~$1$, the free pairs in the new increment tuple $\btau'$ are again pairs $1$ and $2$: the height before
pair~$2$ is still $T_3=2$ and its entries are still $0,-2$. Likewise, the height before pair~$3$ is still~$T_5=0$.
\end{example}

\begin{definition}
On a legal ordering $\btau$ having at least one free pair, we define Move-1 to be the exchange within the earliest\/-\/occurring, first\/-\/met free pair. 
\end{definition}

By Lemma~\ref{LemRemainFree}, the result of Move-1 is legal and has the same free pairs as its argument~$\btau$.

\begin{remark}\label{RemDisjointPairs1}
The pairs of permutations, produced within the set $\Phi_p$ by Move~1, are disjoint: 
if two such pairs intersect, they coincide. Indeed, the double iteration of Move~1 is an identity mapping whenever defined.
\end{remark}

\begin{proposition}
Move-1 is a parity\/-\/reversing involution without fixed points 
on the (sub)set of late\/-\/growing permutations whose
legal increments have 
at least one free pair.

In consequence, the orderings $\btau$ on which Move-1 acts 
cancel out of the sum $\sum_{\Phip}(-)^{\sigma}$. 
\end{proposition}

\begin{example}
In Example~\ref{ex:move1}, the permutation $\sigma=(1,2,8,3,5,\mathbf{4},\mathbf{6},7)$ (parity\/-\/even, $6$ inversions) is paired
with the permutation $\sigma'=(1,2,8,3,5,\mathbf{6},\mathbf{4},7)$ (parity\/-\/odd, $7$ inversions), and the two cancel against each other in the sum of their parities.
\end{example}

\begin{definition}\label{DefRigid}
A legal ordering of increments $\btau$ with no free pair is called \emph{rigid}. 

We claim that 
a rigid ordering waits (the step $0$), then alternates a \emph{payment} $x_i>0$
and a \emph{debt} $y_i<0$, 
and every debt exceeds the height at which its pair started.\footnote{
Let us keep in mind 
that the alternating shape alone is not enough for an increments' ordering $\btau$ to be rigid. 

\textbf{Counterexample.}\quad 
For $p=4$, consider the ordering $\btau=(0,\,3,-2,\,1,-1,\,2,-3)$ that stems from the permutation $\sigma=(1,2,7,4,6,3,8,5)$. 
This ordering of increments alternates the payments
and debts, and it is legal (with the heights $0,3,1,2,1,3,0$), but pair~$2$ is free: it starts at height
$H_1=1$ and its debt is only $-1$, so $T_3+\tau_5=1-1=0\geqslant 0$: the debt could be taken first. Move~1 pairs this ordering $\btau$ with $(0,\,3,-2,\,-1,1,\,2,-3)$.}
\end{definition}

\begin{corollary}
\begin{equation}\label{eq:rigid-sum}
  \sum_{\sigma\in\Phip}(-)^{\sigma}\;=\;\sum_{\btau\ \mathrm{rigid}}(-)^{\sigma}.
\end{equation}
\end{corollary}

\begin{lemma}\label{LemStableFree}
If the zero increment is part of a pair, $\tau_{k>1}=0$ occurring later than the first increment $\tau_1$ which is not part of any pair in a legal ordering 
$\btau = \bigl(\tau_1$,\ $(\tau_2,\tau_3)$,\ $\ldots$,\ $(\tau_{2m},\tau_{2m+1})\bigr)$, then this pair containing $\tau_k=0$ is free.
\end{lemma}

\begin{proof}
Indeed, the height $T$ immediately preceding the pair --\,containing $\tau_k=0$ and the other increment $\tau_\ell$ with $\ell=k-1$ or $\ell=k+1$\,-- is such that the ordered tuple $\btau$ is legal: $T\geqslant 0$ and $T+\tau_k+\tau_\ell\geqslant0$. Yet the height after this entire pair, $T+0+\tau_\ell = T+\tau_\ell+0 \geqslant 0$, also equals the height in the (old or new) ordered increments' tuple such that the other increment $\tau_\ell\neq 0$ precedes $\tau_k=0$ (if it went last) or follows $\tau_k=0$ (if $\tau_\ell$ went earlier).
In brief, the swap of zero step and any other legal step does not make the increments' tuple illegal, hence the pair with $\tau_k$ is free.
\end{proof}

Consequently, all the permutations $\sigma\in\Phi_p$ in which the zero step is not the first, $\tau_1 \neq 0 = \tau_{k>1}$, disappear from the reduced (to only `rigid' increment tuples) summation in the r.\/-\/h.\,s. of~\eqref{eq:rigid-sum}.

\begin{proposition}[the shape of a rigid ordering]\label{prop:shape}
A rigid ordering of increments has the 
shape (recall $m=p-1$):
\begin{equation}\label{EqFormRigid}
  \btau\;=\;(0,\ x_1,\,y_1,\ x_2,\,y_2,\ \dots,\ x_m,\,y_m),
\end{equation}
where $x_1,\dots,x_m$ is an arrangement of $1,\dots,m$ and $-y_1,\dots,-y_m$ 
an arrangement of~$1$,\ $\dots$,\ $m$. 

Denote by 
$H_i \mathrel{{:}{=}} T_{2i+1}$ 
the height after the $i$-th pair in a rigid ordering~$\btau$,
so $H_0=0$ and $H_i=H_{i-1}+x_i+y_i$ for $i>0$. 
Now, for every $i>0$, the entries of~$\btau$ 
satisfy two inequalities:
\begin{equation}\label{eq:R}
  \text{(R1)}\qquad -y_i \;>\; H_{i-1}, \qquad\qquad \text{(R2)}\qquad H_i\;\geqslant\;0 .
\end{equation}
Conversely, every ordering of the form as in~\eqref{EqFormRigid} and satisfying \eqref{eq:R} is legal and rigid.
\end{proposition}

\begin{example}
Before the proof of Proposition~\ref{prop:shape}, a good picture to have in mind (for $p=4$ in particular) is this:
\begin{equation}\label{EqPictureInMind}
  \btau=(\,\underbrace{0}_{\text{wait}},\ \overbrace{\underbrace{2}_{x_1},\underbrace{-1}_{y_1}}^{\text{pair }1},\
  \overbrace{\underbrace{3}_{x_2},\underbrace{-3}_{y_2}}^{\text{pair }2},\
  \overbrace{\underbrace{1}_{x_3},\underbrace{-2}_{y_3}}^{\text{pair }3}\,),
  \qquad
  \begin{array}{c}
  H_0=0,\ H_1=1,\ H_2=1,\ H_3=0,\\[2pt]
  \text{(R1)}:\ 1>0,\ \ 3>1,\ \ 2>1 .
  \end{array}
\end{equation}
This ordered increments tuple is produced from the permutation $\sigma=(1,3,6,2,8,4,7,5)$, and it is rigid: no pair is free because in each pair
the debt $y_i$ is larger in size than the height $H_{i-1}$ the pair started from, so the debt cannot be taken first.
\end{example}

\begin{proof}
Let an increments tuple $\btau$ be rigid, and put $x_i=\tau_{2i}$, $y_i=\tau_{2i+1}$. Pair $i$ is not free, so we negate \eqref{eq:free} to get \(T_{2i-1}+y_i<0\), i.e.\ \(y_i<-T_{2i-1}\). Furthermore, \(T_{2i-1}\geqslant 0\) because \(\btau\) is legal. Thus, $y_i<-T_{2i-1}\leqslant 0$. The walk is legal at the end of the pair, so \(T_{2i+1}\geqslant 0\). Since \(T_{2i+1} = T_{2i-1} + x_i + y_i\), we deduce that 
\[
x_i = T_{2i+1} - T_{2i-1} - y_i \geqslant 0 - T_{2i-1} - y_i = -(T_{2i-1} + y_i).
\]
Thus the $2m$ entries in the pairs are
nonzero; they are all the nonzero elements of the centered set $V_p$
in~\eqref{EqCenteredSet}, so slot $1$ holds $0$, the $x_i$ are
the positive elements and the $y_i$ the negative ones. With $\tau_1=0$ we have $H_0=0$ and $T_{2i-1}=H_{i-1}$, and (R1) is the statement that the $i$th pair is not free while (R2) is the legality at the odd slots. 

Conversely, for an ordering of the stated form~\eqref{EqFormRigid}, condition (R1) says no
pair is free, condition (R2) gives the legality at the odd slots, and legality at the even slots, $T_{2i}=H_{i-1}+x_i\geqslant 0$, is automatic.
\end{proof}

\begin{example}[cf.\ Tables~\ref{tab:moves} and~\ref{tab:moves-full}]\label{ExAfterM2p234}
For $p=2$, we recall from~\cite{cpWron,Method} 
the explicit description of the set $\Phi_2=\{(1,2,3,4)$,\ $(1,2,4,3)$,\ $(1,3,2,4)\}$, 
whence their increment orderings are 
$\btau=(1,0,-1)$, $(0,1,-1)$, and $(1,-1,0)$ respectively.
In the first ordering, the only\/-\/available 
pair is free ($T_1+\tau_3=0$), and exchanging it gives the third ordering; these two cancel in the sum of parities. The second ordering is rigid; it is the staircase.

\smallskip\noindent$\bullet$\quad
For $p=3$ (see Appendix~\ref{AppPhi3} on p.~\pageref{AppPhi3}), Move~1 does most of the job: it pairs $32$ free permutations out of $35$ in~$\Phi_3$. The 
three rigid orderings $\btau$ that remain after Move~1 are
$(0,1,-1,2,-2)$, $(0$,\ $2$,\ $-1$,\ $1$,\ $-2)$, and $(0,2,-2,1,-1)$; 
their `parent' permutations are 
$\sigma=(1,2,6,3,5,4)$, $(1$,\ $2$,\ $5$,\ $3$,\ $6$,\ $4)$, and $(1,3,5,2,6,4)$ respectively. 

\smallskip\noindent$\bullet$\quad
For $p=4$, Move~1 leaves $15$
rigid orderings out of $1001$; these fifteen permutations 
are listed in Table~\ref{tab:rigid4} on p.~\pageref{tab:rigid4} below.
\end{example}

\section{Move 2: exchange two debts under the largest payment}\label{sec:move2}
\noindent%
Before we introduce 
the second move, we define a \emph{locked tail} and state two useful lemmas about rigid orderings.

\begin{definition}\label{def:locked}
Let $\btau$ be rigid and $0\leqslant k\leqslant m$. The pairs $k+1,\dots,m$ of $\btau$ form a
\emph{locked tail} if $x_j=j$ and $y_j=-j$ for all $j>k$, that is, if
$\btau=(\,\dots,\ k{+}1,-(k{+}1),\ \dots,\ m,-m\,)$ ends as the staircase does.
\end{definition}

\begin{example}
For $p=5$ (so $m=4$), the rigid increments' tuple
\[
  \btau=(\,0,\ 2,-1,\ 1,-2,\ \overbrace{3,-3,\ 4,-4}^{\text{locked tail}}\,)
\]
of the permutation $\sigma=(1,2,10,3,9,4,7,5,8,6)$ has a locked tail formed by its pairs $3$ and $4$, with $k=2$; the walk is at height $0$ before each of these two pairs. The tail does not extend further left, since $x_2=1\neq 2$.
\end{example}

\begin{lemma}[a locked tail]\label{lem:tail}
Let a tuple $\btau$ be rigid and suppose $x_j=j$ for all $j>k$, for some $k$ in the range
$0\leqslant k\leqslant m=p-1$. Then the pairs $k+1,\dots,m$ form a locked tail, and
$H_{j-1}=0$ for all $j>k$.
\end{lemma}
\begin{proof}
Downward induction on $j$, starting from $H_m=T_{N-1}=0$. Let $j>k$ and assume $H_j=0$
and $y_\ell=-\ell$ for $\ell>j$. Then the debts $-y_{j+1},\dots,-y_m$ are $j+1,\dots,m$, so
$-y_j\leqslant j$. From $0=H_j=H_{j-1}+x_j+y_j=H_{j-1}+j+y_j$ we get $-y_j=H_{j-1}+j\geqslant j$.
Thus $-y_j=j$ and $H_{j-1}=0$.
\end{proof}

\begin{lemma}[exchanging two debts]\label{lem:debts}
Let an increment tuple $\btau$ be rigid. For some $i$ in the range $1\leqslant i<m$,  suppose $x_i\geqslant -y_i$ and $x_i\geqslant -y_{i+1}$. Then
the ordering obtained by exchanging $y_i$ and $y_{i+1}$ is rigid.
\end{lemma}

\begin{proof}
The new ordering has the shape~\eqref{EqFormRigid} from Proposition~\ref{prop:shape}, 
and only one height changes to reach it:
$H_i$ becomes $H_i'=H_{i-1}+x_i+y_{i+1}$. We check conditions~\eqref{eq:R} for pairs $i$ and $i+1$;
all other instances of \eqref{eq:R} involve only unchanged quantities.
\begin{itemize}
\item (R1) at $i$: $-y_{i+1}>H_{i-1}$. Indeed $-y_{i+1}>H_i$ by (R1) at $i+1$ for
      $\btau$, and $H_i=H_{i-1}+x_i+y_i\geqslant H_{i-1}$ because $x_i\geqslant -y_i$ by our initial assumption.
\item (R2) at $i$: $H_i'=H_{i-1}+x_i+y_{i+1}\geqslant 0$ because $H_{i-1}\geqslant 0$ and $x_i\geqslant -y_{i+1}$ by another initial assumption.
\item (R1) at $i+1$: $-y_i>H_i'$ means $-y_i-y_{i+1}>H_{i-1}+x_i$ (by definition of \(H'_i\)). Again using \(H_i = H_{i-1} + x_i + y_i\implies H_i-  y_i = H_{i-1} + x_i\), we get
      $-y_{i+1}>H_i$, which is (R1) at $i+1$ for~$\btau$.
\item (R2) at $i+1$: $H_{i+1}'=H_i'+x_{i+1}+y_i=H_{i-1} + x_i + y_{i+1}+x_{i+1}+y_i=H_{i+1}\geqslant 0$. \qedhere
\end{itemize}
\end{proof}

\begin{definition}\label{DefMove2}
Move~2 is the exchange of two debts $y_i \rightleftarrows y_{i+1}$ in a rigid increments' tuple $\btau$ that satisfies the assumptions of Lemma~\ref{lem:debts};
the exact choice of the index~$i$ will be specified 
by using Eq.~\eqref{EqWhichPaymentBig} in the proof of our main theorem.
\end{definition}

\begin{example}\label{ex:debts}
Apply Move~2 
with $i=2$ to the rigid tuple
$\btau=(0,\,2,-1,\,3,-3,\,1,-2)$ from~\eqref{EqPictureInMind} on p.~\pageref{EqPictureInMind} above: 
$x_2=3\geqslant -y_2=3$ and $x_2=3\geqslant -y_3=2$, so we may exchange $y_2=-3$ and $y_3=-2$:
\begin{align*}
  \btau =(\,0,\ 2,-1,\ 3,\underline{-3},\ 1,\underline{-2}\,)\ &\longleftrightarrow\
  (\,0,\ 2,-1,\ 3,\underline{-2},\ 1,\underline{-3}\,) = \btau',\\
  H=(0,1,1,0)\ &\longleftrightarrow\ H'=(0,1,2,0).
\end{align*}
Only $H_2$ changed, from $1$ to $2$. The four checks with numbers: (R1) at $2$: $2>1$;
(R2) at $2$: $2\geqslant 0$; (R1) at $3$: $3>2$; (R2) at $3$: $0\geqslant 0$.
In terms of permutations, 
$\sigma=(1,3,6,2,8,4,7,5)$ (parity\/-\/even) becomes $\sigma'=(1,2,6,3,8,4,7,5)$ (parity\/-\/odd): the entries
$\sigma(4)=2$ and $\sigma(2)=3$ have been exchanged, because $y_2$ sits in slot $5$ of
$\btau$, which is position $N+1-5=4$ of $\sigma$, and $y_3$ sits in slot $7$, which is
position~$2$. 
Figure~\ref{fig:move2} illustrates this Move~2 for the tuple~$\btau$ from~\eqref{EqPictureInMind}.
\end{example}
\begin{figure}[htb]
\centering
\begin{tikzpicture}[x=1.15cm,y=0.5cm]
  \fill[gray!12] (4,-0.4) rectangle (7,5.4);
  \draw[walkaxis] (0,0) -- (7.7,0) node[right,font=\scriptsize] {$k$};
  \draw[walkaxis] (0,0) -- (0,5.4) node[above,font=\scriptsize] {$T_k$};
  \foreach \k in {1,...,7} {\draw[gray,thin] (\k,0.1) -- (\k,-0.1) node[below,font=\scriptsize] {$\k$};}
  \draw[ghost] (4,4) -- (5,2) -- (6,3) -- (7,0);
  \draw[walk] (0,0) -- (1,0) -- (2,2) -- (3,1) -- (4,4) -- (5,1) -- (6,2) -- (7,0);
  \foreach \pt in {(0,0),(1,0),(2,2),(3,1),(4,4),(5,1),(6,2),(7,0)} {\node[walkdot] at \pt {};}
  \foreach \pt in {(5,2),(6,3)} {\node[ghostdot] at \pt {};}
  \node[steplabel,below] at (5,0.85) {$H_2=1$};
  \draw[gray,thin] (5.05,2.15) -- (5.45,3.4);
  \node[steplabel,above,gray!80!black] at (5.45,3.4) {$H_2'=2$};
  \node[steplabel,below] at (3,0.85) {$H_1=1$};
  \node[steplabel,right] at (7.05,0.5) {$H_3=0$};
  \node[steplabel] at (2.5,-1.6) {pair $1$};
  \node[steplabel] at (4.5,-1.6) {pair $2$};
  \node[steplabel] at (6.5,-1.6) {pair $3$};
  \draw[gray,thin] (2,-1.1) -- (3,-1.1); \draw[gray,thin] (4,-1.1) -- (5,-1.1); \draw[gray,thin] (6,-1.1) -- (7,-1.1);
  \draw[gray,thin] (2,-1.3) -- (2,-0.9); \draw[gray,thin] (4,-1.3) -- (4,-0.9); \draw[gray,thin] (6,-1.3) -- (6,-0.9); \draw[gray,thin] (7,-1.3) -- (7,-0.9);
  \draw[gray,thin] (3,-1.3) -- (3,-0.9);
  \draw[gray,thin] (5,-1.3) -- (5,-0.9);
\end{tikzpicture}
\caption{Move~2 applied to the rigid ordering of Example~\ref{ex:debts}. Solid: the walk
of $\btau=(0,2,-1,3,-3,1,-2)$. 
Dashed: the walk of its partner $\btau'=(0,2,-1,3,-2,1,-3)$, drawn
where the two differ. The debts of pairs $2$ and $3$ (shaded) are exchanged under the
payment $x_2=3$, which is at least as large as both, and only the height $H_2$ between them
moves.}
\label{fig:move2}
\end{figure}
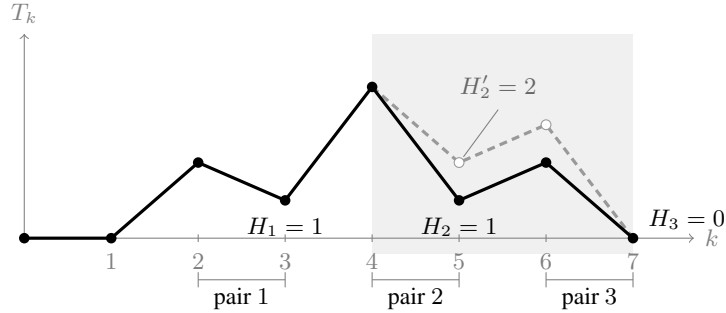

\begin{table}[htb]
\centering\scriptsize\setlength{\tabcolsep}{3.5pt}
\caption{The two moves on $\Phi_p$ for $p\leqslant 6$, with $N=2p$ and $e_p-o_p=\sum_{\sigma\in\Phi_p}(-)^{\sigma}$. 
The table is continued up to $p=14$ in App.~\ref{AppMoves} on p.~\pageref{AppMoves}.}
\label{tab:moves}
\renewcommand{\arraystretch}{1.15}
\begin{tabular}{r r r r r r r c c r}
\toprule
$p$ & $N!$ & $(N-1)!$ & $|\Phi_p|$ & \shortstack{paired off\\by Move~1} & \shortstack{left after\\Move~1 (rigid)} & \shortstack{paired off\\by Move~2} & \shortstack{left after\\Move~2} & $e_p-o_p$ & \shortstack{rigid, in\\$\%$ of $|\Phi_p|$} \\
\midrule
$1$ & $2$ & $1$ & $1$ & $0$ & $1$ & $0$ & $1$ & $+1$ & $100\,\%$ \\
$2$ & $24$ & $6$ & $3$ & $2$ & $1$ & $0$ & $1$ & $-1$ & $33.33\,\%$ \\
$3$ & $720$ & $120$ & $35$ & $32$ & $3$ & $2$ & $1$ & $+1$ & $8.571\,\%$ \\
$4$ & $40\,320$ & $5\,040$ & $1\,001$ & $986$ & $15$ & $14$ & $1$ & $-1$ & $1.499\,\%$ \\
$5$ & $3\,628\,800$ & $362\,880$ & $53\,109$ & $52\,988$ & $121$ & $120$ & $1$ & $+1$ & $0.2278\,\%$ \\
$6$ & $479\,001\,600$ & $39\,916\,800$ & $4\,605\,271$ & $4\,603\,942$ & $1\,329$ & $1\,328$ & $1$ & $-1$ & $0.02886\,\%$ \\
\bottomrule
\end{tabular}
\end{table}

We recall from Definition~\ref{DefStaircase} that 
the \emph{staircase} is the increments' ordering 
$\btau^{*}=(0$,\ $1$,\ $-1$,\ $2$,\ $-2$,\ $\dots$,\ $m$,\ $-m)$ in which
$x_j=j$ for all $j$; it is rigid ($H_i=0$ for all $i$, and $-y_i=i>0$), 
and it is the increments tuple 
of the staircase permutation~$\sigma^{*}$: 
reading $\sigma^{*}$ from the right and subtracting $p+1$ gives
$0,1,-1,2,-2,\dots$. By Lemma~\ref{lem:tail} with $k=0$, the staircase is the \emph{only}
rigid ordering with $x_j=j$ for all $j$: the payments determine the debts to follow immediately.

\begin{proof}[Proof of the main Theorem~\ref{thm:parity} (Part I: why $(-)^{\sigma^*}$)]   
Let $\btau$ be rigid and different from the staircase, and let
\begin{equation}\label{EqWhichPaymentBig}
  k \;=\; \max\{\,j : x_j\ne j\,\} \;\geqslant\;1 .
\end{equation}
By Lemma~\ref{lem:tail}, $y_j=-j$ for $j>k$
(\textbf{NB:} this set $\{j\}$ of indices can be empty, but only for the staircase increments $\btau^*$); 
so the pairs $1,\dots,k$ carry exactly the
payments $1,\dots,k$ and the debts $-1,\dots,-k$. The payment $k$ sits in some pair $i\le k$,
and $i\ne k$ because $x_k\ne k$; so $i<k\leqslant m$. Every debt in the pairs $1,\dots,k$ is at
most $k=x_i$, in particular $-y_i\leqslant x_i$ and $-y_{i+1}\leqslant x_i$ (as $i+1\leqslant k$), so
Lemma~\ref{lem:debts} applies: \emph{exchange $y_i$ and $y_{i+1}$} by Move~2. 
The result is rigid; it
has the same payments in the same slots, thus the same $k$ and the same $i$, so applying
the rule to it exchanges the same two slots and 
restores the increment tuple~$\btau$.
Performed for these choices of $k$ and $i$,
Move~2 is a parity\/-\/reversing
involution without fixed points on the rigid increment orderings 
other than the staircase. Thus
\[
  \sum_{\btau\ \mathrm{rigid}} (-)^{\sigma}\;=\; (-)^{\sigma^{*}} .
\]
This proves the first equality in Theorem~\ref{thm:parity}.
\end{proof}

\begin{table}[htb]
\centering\small
\caption{The $15$ rigid increment orderings $\btau$ for $p=4$ (so $m=3$), grouped as Move~2 pairs them. For
each $\btau$, we restore the permutation $\sigma$ and its parity, the index $k=\max\{j:x_j\ne j\}$, the pair $i$
holding the payment $+k$, and the two debts exchanged (those are in boldface within~$\btau$).}
\label{tab:rigid4}
\renewcommand{\arraystretch}{1.15}
\begin{tabular}{l l c c c l}
\toprule
$\btau$ & $\sigma$ & parity & $k$ & $i$ & exchanged \\
\midrule
$(0,\,1,-1,\,2,-2,\,3,-3)$ & $(1,2,8,3,7,4,6,5)=\sigma^{*}$ & odd & none & & staircase, survives \\
\midrule
$(0,\,2,\mathbf{-2},\,1,\mathbf{-1},\,3,-3)$ & $(1,2,8,4,6,3,7,5)$ & odd  & $2$ & $1$ & $y_1,y_2$ \\
$(0,\,2,-1,\,1,-2,\,3,-3)$ & $(1,2,8,3,6,4,7,5)$ & even & $2$ & $1$ & \\
\midrule
$(0,\,1,-1,\,3,\mathbf{-3},\,2,\mathbf{-2})$ & $(1,3,7,2,8,4,6,5)$ & odd  & $3$ & $2$ & $y_2,y_3$ \\
$(0,\,1,-1,\,3,-2,\,2,-3)$ & $(1,2,7,3,8,4,6,5)$ & even & $3$ & $2$ & \\
\midrule
$(0,\,2,-2,\,3,\mathbf{-3},\,1,\mathbf{-1})$ & $(1,4,6,2,8,3,7,5)$ & odd  & $3$ & $2$ & $y_2,y_3$ \\
$(0,\,2,-2,\,3,-1,\,1,-3)$ & $(1,2,6,4,8,3,7,5)$ & even & $3$ & $2$ & \\
\midrule
$(0,\,2,-1,\,3,\mathbf{-3},\,1,\mathbf{-2})$ & $(1,3,6,2,8,4,7,5)$ & even & $3$ & $2$ & $y_2,y_3$ \\
$(0,\,2,-1,\,3,-2,\,1,-3)$ & $(1,2,6,3,8,4,7,5)$ & odd  & $3$ & $2$ & \\
\midrule
$(0,\,3,\mathbf{-3},\,1,\mathbf{-1},\,2,-2)$ & $(1,3,7,4,6,2,8,5)$ & odd  & $3$ & $1$ & $y_1,y_2$ \\
$(0,\,3,-1,\,1,-3,\,2,-2)$ & $(1,3,7,2,6,4,8,5)$ & even & $3$ & $1$ & \\
\midrule
$(0,\,3,\mathbf{-3},\,2,\mathbf{-2},\,1,-1)$ & $(1,4,6,3,7,2,8,5)$ & odd  & $3$ & $1$ & $y_1,y_2$ \\
$(0,\,3,-2,\,2,-3,\,1,-1)$ & $(1,4,6,2,7,3,8,5)$ & even & $3$ & $1$ & \\
\midrule
$(0,\,3,\mathbf{-3},\,2,\mathbf{-1},\,1,-2)$ & $(1,3,6,4,7,2,8,5)$ & even & $3$ & $1$ & $y_1,y_2$ \\
$(0,\,3,-1,\,2,-3,\,1,-2)$ & $(1,3,6,2,7,4,8,5)$ & odd  & $3$ & $1$ & \\
\bottomrule
\end{tabular}
\end{table}

\begin{example}
In Example~\ref{ex:debts} the payments are $\boldsymbol{x}=(2,3,1)$; no $x_j$ equals $j$, so $k=3$, the
payment $+3$ sits in the pair indexed by $i=2<3$, and Move~2 is exactly the exchange carried out there. 

\noindent$\bullet$\quad Another example 
(with a locked tail):  $\btau=(0,\,2,-2,\,1,-1,\,3,-3)$, that is
$\sigma=(1$,\ $2$,\ $8$,\ $4$,\ $6$,\ $3$,\ $7$,\ $5)$, has $\boldsymbol{x}=(2,1,3)$, so $x_3=3$ and $k=2$; Lemma~\ref{lem:tail}
locks the last pair as $(3,-3)$ at height $H_2=0$; the payment $2$ sits in the pair indexed by $i=1$, and
exchanging $y_1=-2$ with $y_2=-1$ gives the increments $\btau'=(0$,\ $2$,\ $-1$,\ $1$,\ $-2$,\ $3$,\ $-3)$; 
the new permutation is
$\sigma'=(1$,\ $2$,\ $8$,\ $3$,\ $6$,\ $4$,\ $7$,\ $5)$. 

Table~\ref{tab:rigid4} lists all $15$ rigid orderings for
$p=4$ with their partners; the seven pairs cancel and the staircase, which is parity\/-\/odd, remains,
giving the value of the overall sum: $(-1)^{4+1}= -1$.
\end{example}

\begin{remark}[Two moves suffice]
After Moves~1 and~2 have paired up permutations in~$\Phi_p$, only the staircase permutation~$\sigma^{*}$ can remain; no third move is needed. 
Indeed, every $\sigma\in\Phi_p$ either has a free pair or it does not. In the former case it is paired off by Move~1; in the latter case it is, by definition, rigid. A rigid ordering~$\btau$ consists of a locked tail (possibly empty; see Lemma~\ref{lem:tail}) at the end and an unlocked part before it, and Move~2 exchanges two debts inside the unlocked part, under its largest payment~$k$ from~\eqref{EqWhichPaymentBig}. This is possible for every rigid~$\btau$ except the staircase~$\btau^{*}$, in which everything after the wait in slot~$1$ is locked. Finally, Move~1 does not change which pairs are free 
(Lemma~\ref{LemRemainFree}), and Move~2 sends a rigid ordering to a rigid ordering with the same payments, so neither move leaves its own case. Thus every $\sigma\neq\sigma^{*}$ in~$\Phi_p$ is paired by exactly one move, exactly once (see Figure~\ref{fig:moves-flow}).
\end{remark}
\begin{figure}[H]
\centering
\definecolor{oeMoveA}{RGB}{18,128,112}
\definecolor{oeMoveB}{RGB}{226,132,12}
\definecolor{oeStar}{RGB}{238,178,0}
\begin{tikzpicture}[x=1cm,y=1cm,
  q/.style={draw=black!60,fill=black!4,rounded corners=3pt,align=center,minimum width=4.9cm,minimum height=1.05cm,inner sep=5pt,font=\small},
  mA/.style={draw=oeMoveA,fill=oeMoveA!10,rounded corners=3pt,align=center,minimum width=4.9cm,minimum height=1.05cm,inner sep=5pt,font=\small},
  mB/.style={draw=oeMoveB,fill=oeMoveB!12,rounded corners=3pt,align=center,minimum width=4.9cm,minimum height=1.05cm,inner sep=5pt,font=\small},
  st/.style={draw=oeStar!80!black,fill=oeStar!18,rounded corners=3pt,align=center,minimum width=4.9cm,minimum height=1.05cm,inner sep=5pt,font=\small},
  ar/.style={->,thick,black!70},
  lab/.style={font=\footnotesize,black!70}]
  \node[q]  (S)  at (0,0)     {take any $\sigma\in\Phi_p$};
  \node[q]  (Q1) at (0,-1.7)  {does $\btau$ have\\ a free pair?};
  \node[mA] (M1) at (6.4,-1.7) {\textbf{Move~1}: exchange the entries\\ of the first free pair};
  \node[q]  (Q2) at (0,-3.6)  {$\btau$ is rigid;\\ is $x_j=j$ for all $j$?};
  \node[mB] (M2) at (6.4,-3.6) {\textbf{Move~2}: exchange the debts\\ $y_i\rightleftarrows y_{i+1}$ under the payment $k$};
  \node[st] (SC) at (0,-5.5)  {staircase $\sigma^{*}$:\\ fully locked, left alone};
  \node[q]  (C)  at (6.4,-5.5) {cancels against its partner\\ of opposite parity};
  \draw[ar] (S) -- (Q1);
  \draw[ar] (Q1) -- node[lab,above] {yes} (M1);
  \draw[ar] (Q1) -- node[lab,right] {no} (Q2);
  \draw[ar] (Q2) -- node[lab,above] {no} (M2);
  \draw[ar] (Q2) -- node[lab,right] {yes} (SC);
  \draw[ar] (M2) -- (C);
  \draw[ar] (M1.east) -- ++(0.45,0) |- (C.east);
\end{tikzpicture}
\caption{Every $\sigma\in\Phi_p$ is either paired by Move~1, or paired by Move~2, or is the staircase permutation~$\sigma^{*}$.}
\label{fig:moves-flow}
\end{figure}
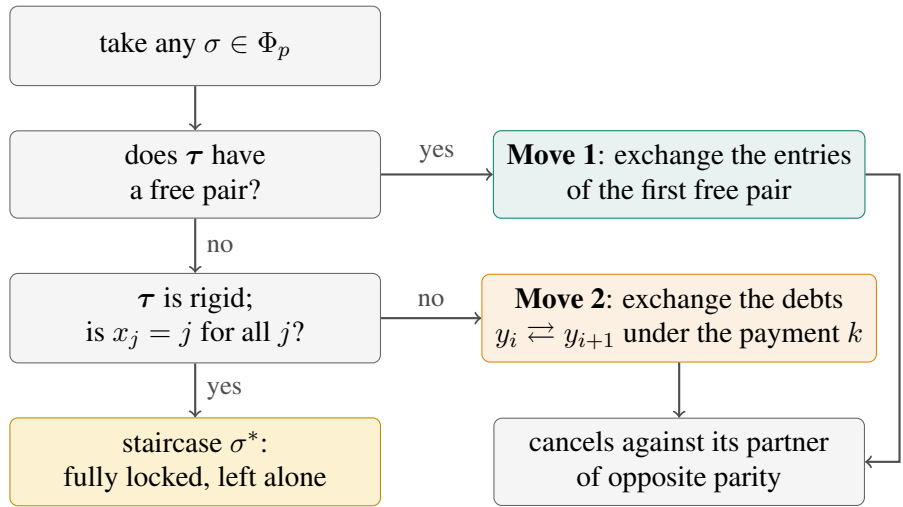

\section{The parity of the staircase permutation $\sigma^{*}$}\label{sec:staircase}
\begin{proof}[Proof of the main Theorem~\ref{thm:parity} (Part II)]   
In the staircase permutation $\sigma^{*}=(1$,\ $2$,\ $2p$,\ $3$,\ $2p-1$,\ $\dots$,\ $p$,\ $p+2$,\ $p+1)$ 
the payment entries $2p,2p-1,\dots,p+2$ occur in decreasing order at the
slots $3,5,\dots,2p-1$; 
the debt entries $1,2,\dots,p$, then neutral $p+1$ occur in
increasing order in slots $1$,\ $2$,\ $4$,\ $\ldots$,\ $2p$.
So the inversions are: all $\binom{p-1}{2}$ pairs among the large
entries, and, for the large entry in slot $2i+1$ (with $1\leqslant i\leqslant p-1$), the $p-i$ smaller entries
$i+2,\dots,p+1$ that come after it. In total
\[
  \inv\sigma^{*}\;=\;\binom{p-1}{2}+\sum_{i=1}^{p-1}(p-i)\;=\;\frac{(p-1)(p-2)}{2}+\frac{(p-1)p}{2}\;=\;(p-1)^{2},
\]
so $(-)^{\sigma^{*}}= (-1)^{(p-1)^2} = (-1)^{p-1}=(-1)^{p+1}$. 
With \eqref{eq:rigid-sum} this completes the proof of 
Theorem~\ref{thm:parity}. 
\end{proof}

\begin{example}
For $p=4$, the staircase permutation is
\[
  \sigma^{*}=(\,1,\ 2,\ \overbrace{8}^{\text{large}},\ 3,\ \overbrace{7}^{\text{large}},\ 4,\
  \overbrace{6}^{\text{large}},\ 5\,).
\]
Its inversions are: large before large, $(8,7),\,(8,6),\,(7,6)$, three of them, which is
$\binom{3}{2}$; and large before small, $(8,3),(8,4),(8,5)$, then $(7,4),(7,5)$, then
$(6,5)$, that is $3+2+1=6$. Nine inversions in all, $9=(4-1)^{2}$, so this staircase permutation $\sigma^{*}$ is parity\/-\/odd,
and $\sum_{\Phi_4}(-)^{\sigma}=-1$, which we knew from Table~\ref{TabPhiP}.
\end{example}

\begin{example}
[$p=3$] 
Of the three rigid orderings of the increments, the one
$\btau=(0$,\ $2$,\ $-1$,\ $1$,\ $-2)$ has the payments $\boldsymbol{x}=(2,1)$, so $k=2$ (since $x_2=1$), the payment $2$ is in pair $1$,
and exchanging $y_1=-1$ with $y_2=-2$ gives the tuple $\btau'=(0$,\ $2$,\ $-2$,\ $1$,\ $-1)$. 
In terms of permutations, the old
$\sigma=(1$,\ $2$,\ $5$,\ $3$,\ $6$,\ $4)$ (which is parity\/-\/odd) cancels against the new, parity\/-\/even permutation $\sigma'=(1$,\ $3$,\ $5$,\ $2$,\ $6$,\ $4)$, and the staircase permutation
$\sigma^{*}=(1$,\ $2$,\ $6$,\ $3$,\ $5$,\ $4)$ (of even parity) remains: $18-17=1 = (-)^{\sigma^*}$, 
as in Table~\ref{TabPhiP}. 
\end{example}

\section{Conclusion}
\noindent%
We seek to understand the structure of the set $\Phi_p\subseteq \BBS_{2p-1}\subsetneq \BBS_{2p}$ of late\/-\/growing permutations, as those contribute to the count of Wronskians and accumulation of the universal constant $\consta(p)$ in the right\/-\/hand side of Eq.~\eqref{EqMain}. In this paper, we established that the set $\Phi_p$ is almost symmetrically balanced by its partitioning $\Phi_p = \Phi_p^{+} \sqcup \Phi_p^{-}$ into parity\/-\/even vs parity\/-\/odd permutations:
$|\Phi_p^{-}| - |\Phi_p^{+}| = (-1)^p$, i.e.\ the two cardinalities differ by only~$\pm1$ (see Table~\ref{TabPhiP} on p.~\pageref{TabPhiP} continued in
Appendix~\ref{AppMoves} as far as $p=14$).
The parity of permutations in the dominant set interlaces by~$(-1)^{p+1}$.

Given this fact, all the more surprising it is that 
the universal constants $\consta(p)$ --\,accumulated with the signs $(-)^\sigma$ of permutations $\sigma\in\Phi_p$ that contribute by the products of falling factorials $E^{\underline{p}}_k$ (see~\cite{cpWron,Method,Valladolid})\,-- are strictly positive at all $p\in\BBN_{\geqslant 1}$, and the values grow as fast as $\consta(p) \sim p^{2p^2}$ as $p\to +\infty$, cf.~\cite{Method} for more details.

This indicates that the set $\Phi_p\subseteq \BBS_{2p-1}$ of late\/-\/growing permutations has a hidden \textbf{\textit{a}}symmetry, in contrast with its almost\/-\/symmetry w.r.t.\ the parity function. 
Specifically, positive (if $(-)^\sigma > 0$) and negative (resp., if $(-)^\sigma < 0$) 
in almost equal number of cases, the contributions from late\/-\/growing permutations $\sigma\in\Phi_p$ accumulate in the positive direction overall: $\consta(p) > 0$, irrespective of the parity of late\/-\/growing permutations in the dominant set $\Phi_p^{\pm}$, which interlaces by~$(-1)^{p+1}$.
To reveal the nature of 
this asymmetry is a 
problem for future research.

\subsection*{Discussion}
We recall that by using the monomial weights of the $p$th order differential operators alternated in the l.\/-\/h.\,s.\ of~\eqref{EqMain}, we gain the interpretation of late\/-\/growing permutations $\sigma\in\Phi_p$ as the constrained monomial degree walks in~$\BBZ_{\geqslant 0}$. 
By construction, the count of monomial degrees starts with $T_0=0$ at $k=0$ and ends by $T_{N-1}=0$ at $k=2p-1$. 
Every such 
permutation contributes $(-)^\sigma$ times the product $c(\sigma,p)=\prod_{k=1}^{N-1} E^{\underline{p}}_k $ of falling factorials $E^{\underline{p}}_k $ (introduced in~\cite{Method}) 
to the constant $\consta(p)$ in the r.\/-\/h.\,s.\ of~\eqref{EqMain}.

Consider only those late\/-\/growing permutations $\sigma\in\Phi_p$ for which 
some \emph{intermediate} degree $T_k$ is pinned to zero at $1\leqslant k \leqslant N-2$. (Let us reiterate that pinning the last degree $T_{2p-1} = 0$ does not constrain the permutations in~$\Phi_p$; the summation with signs over the entire set~$\Phi_p$ yields the value $\consta(p)$ by reaching $k=2p-1$, as explained in~\cite{Method}.)

First, let us 
examine the data in Appendix~\ref{AppPhi3} about the set $\Phi_3$ of late\/-\/growing permutations~$\sigma$ and their contributions $(-)^\sigma\,c(\sigma,3)$ to the r.\/-\/h.\,s.\ of~\eqref{EqMain} at $p=3$. 
In Fig.~\ref{Fig35walks} on p.~\pageref{Fig35walks} we draw all the `landscapes'
$k\mapsto T_k$ (with $0\leqslant k \leqslant 5 = 2p-1$ here).

We discover from Table~\ref{tab:phi3} represented in Fig.~\ref{Fig35walks} that all these partial sums (over the late\/-\/growing permutations such that $T_k=0$)
predict the integral value~$\consta(p=3)$, namely,
\[
\sum_{\substack{ {\sigma\in\Phi_3 \mid T_k=0} \\ 
{1\leqslant k \leqslant 4=2p-2} } 
} (-)^\sigma\,c(\sigma,p=3) = \pm \consta(p=3);
\]
the regular pattern of signs in the right\/-\/hand side of these equalities will presently be discussed.

In hindsight, the equality, up to a sign, of the pinned partial sums to the overall sum $\consta(p)$ (always${}>0$) could be seen for $p=2$ (cf.\ Example~\ref{ExAfterM2p234} on p.~\pageref{ExAfterM2p234} and~\cite[Ex.\:4.1]{Method}): the pinned partial sums 
for singleton\/-\/set permutations $\sigma_2=(1,2,4,3)$ and $\sigma_3=(1,3,2,4)$ were
\[
(-)^{\sigma_2}\, c(\sigma_2,p=2)=-2 \quad \text{and}\quad 
(-)^{\sigma_3}\,c(\sigma_3,p=2) = -2, 
\]
predicting the cumulative value $\consta(p=2) = +2$ up to a sign.

We verify this hidden symmetry of the sets $\Phi_p$ for larger values~$p$:

\begin{proposition}\label{PropPredict}
For every 
integer $p$ such that $1< p \leqslant 14$ and for any step~$k$ from the range $1\leqslant k \leqslant 2p-2$, consider only those late\/-\/growing permutations $\sigma\in\Phi_p\subseteq\BBS_{2p-1}$ such that the monomial degree (or \emph{height}) is pinned to $T_k\mathrel{{:}{=}}0$ at this fixed step~$k$.\footnote{\label{FigPinnedSetsOverlap}
Naturally, such sets can overlap for $p>2$, as seen from Fig.~\ref{Fig35walks} at $p=3$.}
Then the partial sum of contributions $(-)^\sigma\, c(\sigma,p)$ of these selected permutations~$\sigma$ to the integral constant $\consta(p)$ in the r.\/-\/h.\,s.\ of~\eqref{EqMain} predicts the value~$\consta(p)$\textup{:}
\begin{equation}\label{EqPredictValue}
\sum_{\substack{ {\sigma\in\Phi_p \mid T_k=0} \\ 
{1\leqslant k \leqslant 2p-2} } 
} (-)^\sigma\,c(\sigma,p) = \pm \consta(p)
\end{equation}
where the signs in the right\/-\/hand side 
are read off Table~\ref{TabSignsP14}.
\end{proposition}

\begin{table}[!ht]\centering\begin{tikzpicture}
\foreach \k in {1,...,27} {\node[font=\scriptsize] at (\k*\scw,0.62) {\k};}
\node[font=\scriptsize,anchor=east] at (0.32,0.62) {$k$};
\signrow{0}{2}{-,-,+}
\signrow{1}{3}{+,-,-,+,+}
\signrow{2}{4}{-,-,+,+,-,-,+}
\signrow{3}{5}{+,-,-,+,+,-,-,+,+}
\signrow{4}{6}{-,-,+,+,-,-,+,+,-,-,+}
\signrow{5}{7}{+,-,-,+,+,-,-,+,+,-,-,+,+}
\signrow{6}{8}{-,-,+,+,-,-,+,+,-,-,+,+,-,-,+}
\signrow{7}{9}{+,-,-,+,+,-,-,+,+,-,-,+,+,-,-,+,+}
\signrow{8}{10}{-,-,+,+,-,-,+,+,-,-,+,+,-,-,+,+,-,-,+}
\signrow{9}{11}{+,-,-,+,+,-,-,+,+,-,-,+,+,-,-,+,+,-,-,+,+}
\signrow{10}{12}{-,-,+,+,-,-,+,+,-,-,+,+,-,-,+,+,-,-,+,+,-,-,+}
\signrow{11}{13}{+,-,-,+,+,-,-,+,+,-,-,+,+,-,-,+,+,-,-,+,+,-,-,+,+}
\signrow{12}{14}{-,-,+,+,-,-,+,+,-,-,+,+,-,-,+,+,-,-,+,+,-,-,+,+,-,-,+}
\end{tikzpicture}
\caption{Each cell: the sum of the contributions of all $\sigma\in\Phi_p$ with $T_k=0$ (for a fixed $k$), which is always $+\consta(p)$ (blue) or $-\consta(p)$ (orange).}
\label{TabSignsP14}
\end{table}

\begin{conjecture}\label{ConjPredictCp}
Predicting the values $\consta(p)$ via summation over the permutation subsets smaller than~$\Phi_p$, equalities~\eqref{EqPredictValue} are valid for all integers $p\geqslant 2$, with the ornament of signs in the r.\/-\/h.\,s.\ of~\eqref{EqPredictValue} regularly extending Table~\ref{TabSignsP14}: for all admissible steps $k$ 
(with $1\leqslant k \leqslant 2p-1$),
the ornament's period in the horizontal $k$-direction is $4$ and in the vertical $p$-direction is~$2$.
\end{conjecture}

In the update of paper~\cite{Method} we shall measure the gain (w.r.t.\ the computation time and memory allocation) from this shortcut, which yields 
the accelerated Algorithm~3 for calculation of the values~$\consta(p)$ in the new integer sequence OEIS A392714~(\cite{oeisA392714}).

\smallskip\noindent
\textbf{Approbation.}\quad
The conjecture $|\Phi_p^{-}| - |\Phi_p^{+}| = (-1)^p$ (that the subset $\Phi_p\subseteq\BBS_{2p-1}\subsetneq\BBS_{2p}$ is almost symmetric w.r.t.\ its partition $\Phi_p = \Phi_p^{+} \sqcup \Phi_p^{-}$ into 
parity\/-\/even vs odd permutations of $n=2p-1$ objects $\{ 2,\ldots,2p \}$), itself based on the empiric evidence in~\cite{cpWron}, was presented by the first author at the 17th Algorithmic Number Theory symposium -- ANTS17 (6--10 July 2026 in Groningen, NL), at the 36th International colloquium on group theoretic methods in Physics -- \textsc{Group36} (13--17 July 2026 in Valladolid, Spain), and at the annual colloquium in Geometry and Quantum Theory -- GQT (24--28 August 2026 in Malden near Nijmegen, NL).
The second author presented this line of study at the Open Problems session during the 2nd ISNMP Conference on nonlinear mathematical physics (28 June -- 4 July 2026 in Bad Ems, Germany).
The authors thank the organisers of these events for stimulating discussions and (partial) waiver of the registration fees for the first author.

\smallskip\noindent
\textbf{Acknowledgements.}\quad
The authors thank 
M.\,Kontsevich and O.\,Zaboronski for their attention and advice.
The travel of the first author to the \textsc{Group36} and GQT colloquia was supported by the Bernoulli Institute (via projects `Algebra' and `Graduate programme').
The travel of the second author is supported by the Bernoulli Swiss scientific travel fellowship and by the Bernoulli Institute (via project `Algebra').

\newpage
\appendix 
\section{The count of late\/-\/growing permutations, $|\Phi_p|={}$A147681($2p-1$)%
}\label{AppLateGrowLimRatios}
\noindent%
Denote by $n$ the number of permuted objects. 
The integer sequence elements A147681($n$)
count the late\/-\/growing permutations of $n$ objects.
(The values of integer sequence A147681 are taken from the OEIS entry~\cite{oeisA147681}.)
Whenever $n = 2p-1$ is odd, A147681($2p-1) = |\Phi_p|$.
Since $N = 2p$, we have $(n-1)! = (N-2)!$ (as in Table~\ref{TabPhiP} on p.~\pageref{TabPhiP}).

{\scriptsize%
\begin{verbatim}
n = 3  (odd,  |Phi_2|, N = 4)
  numerator    A147681(3) = 3
  denominator  (3-1)! = 2! = 2
  ratio        3 / 2 = 1.500000
  ratio - 1    = 0.500000

n = 5  (odd,  |Phi_3|, N = 6)
  numerator    A147681(5) = 35
  denominator  (5-1)! = 4! = 24
  ratio        35 / 24 = 1.458333
  ratio - 1    = 0.458333

n = 7  (odd,  |Phi_4|, N = 8)
  numerator    A147681(7) = 1001
  denominator  (7-1)! = 6! = 720
  ratio        1001 / 720 = 1.390278
  ratio - 1    = 0.390278

n = 9  (odd,  |Phi_5|, N = 10)
  numerator    A147681(9) = 53109
  denominator  (9-1)! = 8! = 40320
  ratio        53109 / 40320 = 1.317188
  ratio - 1    = 0.317188

n = 11  (odd,  |Phi_6|, N = 12)
  numerator    A147681(11) = 4605271
  denominator  (11-1)! = 10! = 3628800
  ratio        4605271 / 3628800 = 1.269089
  ratio - 1    = 0.269089

n = 13  (odd,  |Phi_7|, N = 14)
  numerator    A147681(13) = 589809987
  denominator  (13-1)! = 12! = 479001600
  ratio        589809987 / 479001600 = 1.231332
  ratio - 1    = 0.231332

n = 15  (odd,  |Phi_8|, N = 16)
  numerator    A147681(15) = 104899483845
  denominator  (15-1)! = 14! = 87178291200
  ratio        104899483845 / 87178291200 = 1.203275
  ratio - 1    = 0.203275

n = 17  (odd,  |Phi_9|, N = 18)
  numerator    A147681(17) = 24698838710457
  denominator  (17-1)! = 16! = 20922789888000
  ratio        24698838710457 / 20922789888000 = 1.180475
  ratio - 1    = 0.180475

n = 19  (odd,  |Phi_10|, N = 20)
  numerator    A147681(19) = 7444522779300351
  denominator  (19-1)! = 18! = 6402373705728000
  ratio        7444522779300351 / 6402373705728000 = 1.162775
  ratio - 1    = 0.162775

n = 21  (odd,  |Phi_11|, N = 22)
  numerator    A147681(21) = 2792467448952670671
  denominator  (21-1)! = 20! = 2432902008176640000
  ratio        2792467448952670671 / 2432902008176640000 = 1.147793
  ratio - 1    = 0.147793

n = 23  (odd,  |Phi_12|, N = 24)
  numerator    A147681(23) = 1276369340371154144337
  denominator  (23-1)! = 22! = 1124000727777607680000
  ratio        1276369340371154144337 / 1124000727777607680000 = 1.135559
  ratio - 1    = 0.135559

n = 25  (odd,  |Phi_13|, N = 26)
  numerator    A147681(25) = 698008560075731437652425
  denominator  (25-1)! = 24! = 620448401733239439360000
  ratio        698008560075731437652425 / 620448401733239439360000 = 1.125007
  ratio - 1    = 0.125007

n = 27  (odd,  |Phi_14|, N = 28)
  numerator    A147681(27) = 450111715263775132783135875
  denominator  (27-1)! = 26! = 403291461126605635584000000
  ratio        450111715263775132783135875 / 403291461126605635584000000 = 1.116095
  ratio - 1    = 0.116095
\end{verbatim}

}

\medskip\noindent%
The excess $|\Phi_p|/(2p-2)! - 1$ tends to~$0$ from above as~$O\bigl(\log(p)/p\bigr)$ as $p\to+\infty$.

\newpage
\section{The pairings in $\Phi_3$ by Moves 1 and 2}\label{AppPhi3}
\begin{table}[H]
\centering
\footnotesize
\setlength{\tabcolsep}{4pt}
\renewcommand{\arraystretch}{1.0}
\begin{tabular}{rcccc r c c r}
\toprule
$\#$ & $\sigma$ & $\btau$ & $(T_1,\dots,T_5)$ & $(-)^\sigma$ & contribution & pair & $\btau$-slots & pair sum \\
\midrule
\multicolumn{9}{l}{\textit{Move~1: pairs that differ by exchanging the first free pair}} \\
1 & $123456$ & $(2,1,0,-1,-2)$ & $(2,3,3,2,0)$ & $+$ & $9\,000$ & \multirow{2}{*}{F1} & \multirow{2}{*}{$2\leftrightarrow3$} & \multirow{2}{*}{$4\,500$} \\
2 & $123546$ & $(2,0,1,-1,-2)$ & $(2,2,3,2,0)$ & $-$ & $-4\,500$ &  &  &  \\
\addlinespace[1pt]
3 & $123564$ & $(0,2,1,-1,-2)$ & $(0,2,3,2,0)$ & $+$ & $450$ & \multirow{2}{*}{F2} & \multirow{2}{*}{$2\leftrightarrow3$} & \multirow{2}{*}{$270$} \\
4 & $123654$ & $(0,1,2,-1,-2)$ & $(0,1,3,2,0)$ & $-$ & $-180$ &  &  &  \\
\addlinespace[1pt]
5 & $123645$ & $(1,0,2,-1,-2)$ & $(1,1,3,2,0)$ & $+$ & $720$ & \multirow{2}{*}{F3} & \multirow{2}{*}{$2\leftrightarrow3$} & \multirow{2}{*}{$-2\,880$} \\
6 & $123465$ & $(1,2,0,-1,-2)$ & $(1,3,3,2,0)$ & $-$ & $-3\,600$ &  &  &  \\
\addlinespace[1pt]
7 & $124365$ & $(1,2,-1,0,-2)$ & $(1,3,2,2,0)$ & $+$ & $1\,800$ & \multirow{2}{*}{F4} & \multirow{2}{*}{$2\leftrightarrow3$} & \multirow{2}{*}{$1\,710$} \\
8 & $124635$ & $(1,-1,2,0,-2)$ & $(1,0,2,2,0)$ & $-$ & $-90$ &  &  &  \\
\addlinespace[1pt]
9 & $124536$ & $(2,-1,1,0,-2)$ & $(2,1,2,2,0)$ & $+$ & $900$ & \multirow{2}{*}{F5} & \multirow{2}{*}{$2\leftrightarrow3$} & \multirow{2}{*}{$-3\,600$} \\
10 & $124356$ & $(2,1,-1,0,-2)$ & $(2,3,2,2,0)$ & $-$ & $-4\,500$ &  &  &  \\
\addlinespace[1pt]
11 & $125346$ & $(2,0,-1,1,-2)$ & $(2,2,1,2,0)$ & $+$ & $900$ & \multirow{2}{*}{F6} & \multirow{2}{*}{$2\leftrightarrow3$} & \multirow{2}{*}{$540$} \\
12 & $125436$ & $(2,-1,0,1,-2)$ & $(2,1,1,2,0)$ & $-$ & $-360$ &  &  &  \\
\addlinespace[1pt]
13 & $126435$ & $(1,-1,0,2,-2)$ & $(1,0,0,2,0)$ & $+$ & $9$ & \multirow{2}{*}{F7} & \multirow{2}{*}{$2\leftrightarrow3$} & \multirow{2}{*}{$-27$} \\
14 & $126345$ & $(1,0,-1,2,-2)$ & $(1,1,0,2,0)$ & $-$ & $-36$ &  &  &  \\
\addlinespace[1pt]
15 & $132465$ & $(1,2,0,-2,-1)$ & $(1,3,3,1,0)$ & $+$ & $1\,440$ & \multirow{2}{*}{F8} & \multirow{2}{*}{$2\leftrightarrow3$} & \multirow{2}{*}{$1\,152$} \\
16 & $132645$ & $(1,0,2,-2,-1)$ & $(1,1,3,1,0)$ & $-$ & $-288$ &  &  &  \\
\addlinespace[1pt]
17 & $132546$ & $(2,0,1,-2,-1)$ & $(2,2,3,1,0)$ & $+$ & $1\,800$ & \multirow{2}{*}{F9} & \multirow{2}{*}{$2\leftrightarrow3$} & \multirow{2}{*}{$-1\,800$} \\
18 & $132456$ & $(2,1,0,-2,-1)$ & $(2,3,3,1,0)$ & $-$ & $-3\,600$ &  &  &  \\
\addlinespace[1pt]
19 & $132654$ & $(0,1,2,-2,-1)$ & $(0,1,3,1,0)$ & $+$ & $72$ & \multirow{2}{*}{F10} & \multirow{2}{*}{$2\leftrightarrow3$} & \multirow{2}{*}{$-108$} \\
20 & $132564$ & $(0,2,1,-2,-1)$ & $(0,2,3,1,0)$ & $-$ & $-180$ &  &  &  \\
\addlinespace[1pt]
21 & $134256$ & $(2,1,-2,0,-1)$ & $(2,3,1,1,0)$ & $+$ & $720$ & \multirow{2}{*}{F11} & \multirow{2}{*}{$2\leftrightarrow3$} & \multirow{2}{*}{$684$} \\
22 & $134526$ & $(2,-2,1,0,-1)$ & $(2,0,1,1,0)$ & $-$ & $-36$ &  &  &  \\
\addlinespace[1pt]
23 & $135426$ & $(2,-2,0,1,-1)$ & $(2,0,0,1,0)$ & $+$ & $9$ & \multirow{2}{*}{F12} & \multirow{2}{*}{$2\leftrightarrow3$} & \multirow{2}{*}{$-81$} \\
24 & $135246$ & $(2,0,-2,1,-1)$ & $(2,2,0,1,0)$ & $-$ & $-90$ &  &  &  \\
\addlinespace[1pt]
25 & $142356$ & $(2,1,-1,-2,0)$ & $(2,3,2,0,0)$ & $+$ & $450$ & \multirow{2}{*}{F13} & \multirow{2}{*}{$2\leftrightarrow3$} & \multirow{2}{*}{$360$} \\
26 & $142536$ & $(2,-1,1,-2,0)$ & $(2,1,2,0,0)$ & $-$ & $-90$ &  &  &  \\
\addlinespace[1pt]
27 & $142635$ & $(1,-1,2,-2,0)$ & $(1,0,2,0,0)$ & $+$ & $9$ & \multirow{2}{*}{F14} & \multirow{2}{*}{$2\leftrightarrow3$} & \multirow{2}{*}{$-171$} \\
28 & $142365$ & $(1,2,-1,-2,0)$ & $(1,3,2,0,0)$ & $-$ & $-180$ &  &  &  \\
\addlinespace[1pt]
29 & $143526$ & $(2,-2,1,-1,0)$ & $(2,0,1,0,0)$ & $+$ & $9$ & \multirow{2}{*}{F15} & \multirow{2}{*}{$2\leftrightarrow3$} & \multirow{2}{*}{$-171$} \\
30 & $143256$ & $(2,1,-2,-1,0)$ & $(2,3,1,0,0)$ & $-$ & $-180$ &  &  &  \\
\addlinespace[1pt]
31 & $143265$ & $(1,2,-2,-1,0)$ & $(1,3,1,0,0)$ & $+$ & $72$ & \multirow{2}{*}{F16} & \multirow{2}{*}{$4\leftrightarrow5$} & \multirow{2}{*}{$-216$} \\
32 & $134265$ & $(1,2,-2,0,-1)$ & $(1,3,1,1,0)$ & $-$ & $-288$ &  &  &  \\
\addlinespace[1pt]
\cmidrule(lr){1-9}
\multicolumn{8}{l}{\textit{Move~1 subtotal (free)}} & $162$ \\
\midrule
\multicolumn{9}{l}{\textit{Move~2: rigid orderings, two debts exchanged}} \\
33 & $135264$ & $(0,2,-2,1,-1)$ & $(0,2,0,1,0)$ & $+$ & $9$ & \multirow{2}{*}{R1} & \multirow{2}{*}{$3\leftrightarrow5$} & \multirow{2}{*}{$-81$} \\
34 & $125364$ & $(0,2,-1,1,-2)$ & $(0,2,1,2,0)$ & $-$ & $-90$ &  &  &  \\
\addlinespace[1pt]
\cmidrule(lr){1-9}
\multicolumn{8}{l}{\textit{Move~2 subtotal (rigid)}} & $-81$ \\
\midrule
\multicolumn{9}{l}{\textit{Staircase: no partner}} \\
35 & $126354$ & $(0,1,-1,2,-2)$ & $(0,1,0,2,0)$ & $+$ & $9$ & S & & $9$ \\
\midrule
\multicolumn{8}{l}{$\sum$ of all contributions $=\consta(3)$} & $90$ \\
\multicolumn{8}{l}{$\sum$ of the positive contributions} & $18\,378$ \\
\multicolumn{8}{l}{$\sum$ of the negative contributions} & $-18\,288$ \\
\bottomrule
\end{tabular}
\caption{All $35$ elements of $\Phi_3$ (out of the $6!=720$ permutations $\sigma\in\BBS_6$), grouped as Moves~1 and~2 pair them: $16$ pairs F1--F16 by Move~1, one pair R1 of rigid orderings by Move~2, and the staircase permutation $\sigma^*$ (S). Within each pair the parity\/-\/even permutation is listed first; the column `$\btau$-slots' shows which two slots of $\btau$ are exchanged, and `pair sum' adds the two contributions to~$\consta(3)$.}
\label{tab:phi3}
\end{table}

\begin{landscape}
\section{The two moves on $\Phi_p$ for $p\leqslant 14$}\label{AppMoves}
\begin{table}[htb]
\centering\scriptsize\setlength{\tabcolsep}{2.6pt}
\caption{For $p\leqslant 14$ here, $\Phi_p$ is the set of late\/-\/growing permutations of $2p-1$ objects $2$,\ $\ldots$,\ $N=2p$.
We report the counts of permutations paired off by Move~1 (in such, there was at least one free pair), then by Move~2 (rigid), except the staircase permutation $\sigma^*$ that always survives till the end, marking the dominant subset in~$\Phi_p$:
of either parity\/-\/even ($e_p$ many) or parity\/-\/odd (their number is denoted by~$o_p$).
Approximate values of $(2p)!$ and $(2p-1)!$ are given if $p\geqslant7$.}
\label{tab:moves-full}
\renewcommand{\arraystretch}{1.25}
\begin{tabular}{r r r r r r r c c r}
\toprule
$p$ & $N!$ & $(N-1)!$ & $|\Phi_p|$ & \shortstack{paired off\\by Move~1} & \shortstack{left after\\Move~1 (rigid)} & \shortstack{paired off\\by Move~2} & \shortstack{left after\\Move~2} & $e_p-o_p$ & \shortstack{rigid, in\\$\%$ of $|\Phi_p|$} \\
\midrule
$1$ & $2$ & $1$ & $1$ & $0$ & $1$ & $0$ & $1$ & $+1$ & $100\,\%$ \\
$2$ & $24$ & $6$ & $3$ & $2$ & $1$ & $0$ & $1$ & $-1$ & $33.33\,\%$ \\
$3$ & $720$ & $120$ & $35$ & $32$ & $3$ & $2$ & $1$ & $+1$ & $8.571\,\%$ \\
$4$ & $40\,320$ & $5\,040$ & $1\,001$ & $986$ & $15$ & $14$ & $1$ & $-1$ & $1.499\,\%$ \\
$5$ & $3\,628\,800$ & $362\,880$ & $53\,109$ & $52\,988$ & $121$ & $120$ & $1$ & $+1$ & $0.2278\,\%$ \\
$6$ & $479\,001\,600$ & $39\,916\,800$ & $4\,605\,271$ & $4\,603\,942$ & $1\,329$ & $1\,328$ & $1$ & $-1$ & $0.02886\,\%$ \\
$7$ & $8.72\cdot 10^{10}$ & $6.23\cdot 10^{9}$ & $589\,809\,987$ & $589\,789\,720$ & $20\,267$ & $20\,266$ & $1$ & $+1$ & $3.44\cdot 10^{-3}\,\%$ \\
$8$ & $2.09\cdot 10^{13}$ & $1.31\cdot 10^{12}$ & $104\,899\,483\,845$ & $104\,899\,090\,278$ & $393\,567$ & $393\,566$ & $1$ & $-1$ & $3.75\cdot 10^{-4}\,\%$ \\
$9$ & $6.40\cdot 10^{15}$ & $3.56\cdot 10^{14}$ & $24\,698\,838\,710\,457$ & $24\,698\,828\,926\,664$ & $9\,783\,793$ & $9\,783\,792$ & $1$ & $+1$ & $3.96\cdot 10^{-5}\,\%$ \\
$10$ & $2.43\cdot 10^{18}$ & $1.22\cdot 10^{17}$ & $7\,444\,522\,779\,300\,351$ & $7\,444\,522\,482\,068\,318$ & $297\,232\,033$ & $297\,232\,032$ & $1$ & $-1$ & $3.99\cdot 10^{-6}\,\%$ \\
$11$ & $1.12\cdot 10^{21}$ & $5.11\cdot 10^{19}$ & $2\,792\,467\,448\,952\,670\,671$ & $2\,792\,467\,437\,906\,312\,444$ & $11\,046\,358\,227$ & $11\,046\,358\,226$ & $1$ & $+1$ & $3.96\cdot 10^{-7}\,\%$ \\
$12$ & $6.20\cdot 10^{23}$ & $2.59\cdot 10^{22}$ & $1\,276\,369\,340\,371\,154\,144\,337$ & $1\,276\,369\,339\,884\,584\,175\,906$ & $486\,569\,968\,431$ & $486\,569\,968\,430$ & $1$ & $-1$ & $3.81\cdot 10^{-8}\,\%$ \\
$13$ & $4.03\cdot 10^{26}$ & $1.55\cdot 10^{25}$ & $698\,008\,560\,075\,731\,437\,652\,425$ & $698\,008\,560\,050\,375\,355\,946\,336$ & $25\,356\,081\,706\,089$ & $25\,356\,081\,706\,088$ & $1$ & $+1$ & $3.63\cdot 10^{-9}\,\%$ \\
$14$ & $3.05\cdot 10^{29}$ & $1.09\cdot 10^{28}$ & $450\,111\,715\,263\,775\,132\,783\,135\,875$ & $450\,111\,715\,262\,245\,154\,492\,655\,218$ & $1\,529\,978\,290\,480\,657$ & $1\,529\,978\,290\,480\,656$ & $1$ & $-1$ & $3.40\cdot 10^{-10}\,\%$ \\
\bottomrule
\end{tabular}
\end{table}
\end{landscape}
%

\section{The monomial degree walks for $35$ 
permutations in~$\Phi_3$}\label{App35walks}
\begin{figure}[H]\centering
\renewcommand{\pinK}{0}
\begin{tikzpicture}
\walkpanel{0}{0}{123456}{2,3,3,2,0}{+9\,000}{up}{1}
\walkpanel{1}{0}{123546}{2,2,3,2,0}{-4\,500}{dn}{1}
\walkpanel{2}{0}{123564}{0,2,3,2,0}{+450}{up}{1}
\walkpanel{3}{0}{123654}{0,1,3,2,0}{-180}{dn}{1}
\walkpanel{4}{0}{123645}{1,1,3,2,0}{+720}{up}{1}
\walkpanel{0}{1}{123465}{1,3,3,2,0}{-3\,600}{dn}{1}
\walkpanel{1}{1}{124365}{1,3,2,2,0}{+1\,800}{up}{1}
\walkpanel{2}{1}{124635}{1,0,2,2,0}{-90}{dn}{1}
\walkpanel{3}{1}{124536}{2,1,2,2,0}{+900}{up}{1}
\walkpanel{4}{1}{124356}{2,3,2,2,0}{-4\,500}{dn}{1}
\walkpanel{0}{2}{125346}{2,2,1,2,0}{+900}{up}{1}
\walkpanel{1}{2}{125436}{2,1,1,2,0}{-360}{dn}{1}
\walkpanel{2}{2}{126435}{1,0,0,2,0}{+9}{up}{1}
\walkpanel{3}{2}{126345}{1,1,0,2,0}{-36}{dn}{1}
\walkpanel{4}{2}{132465}{1,3,3,1,0}{+1\,440}{up}{1}
\walkpanel{0}{3}{132645}{1,1,3,1,0}{-288}{dn}{1}
\walkpanel{1}{3}{132546}{2,2,3,1,0}{+1\,800}{up}{1}
\walkpanel{2}{3}{132456}{2,3,3,1,0}{-3\,600}{dn}{1}
\walkpanel{3}{3}{132654}{0,1,3,1,0}{+72}{up}{1}
\walkpanel{4}{3}{132564}{0,2,3,1,0}{-180}{dn}{1}
\walkpanel{0}{4}{134256}{2,3,1,1,0}{+720}{up}{1}
\walkpanel{1}{4}{134526}{2,0,1,1,0}{-36}{dn}{1}
\walkpanel{2}{4}{135426}{2,0,0,1,0}{+9}{up}{1}
\walkpanel{3}{4}{135246}{2,2,0,1,0}{-90}{dn}{1}
\walkpanel{4}{4}{142356}{2,3,2,0,0}{+450}{up}{1}
\walkpanel{0}{5}{142536}{2,1,2,0,0}{-90}{dn}{1}
\walkpanel{1}{5}{142635}{1,0,2,0,0}{+9}{up}{1}
\walkpanel{2}{5}{142365}{1,3,2,0,0}{-180}{dn}{1}
\walkpanel{3}{5}{143526}{2,0,1,0,0}{+9}{up}{1}
\walkpanel{4}{5}{143256}{2,3,1,0,0}{-180}{dn}{1}
\walkpanel{0}{6}{143265}{1,3,1,0,0}{+72}{up}{1}
\walkpanel{1}{6}{134265}{1,3,1,1,0}{-288}{dn}{1}
\walkpanel{2}{6}{135264}{0,2,0,1,0}{+9}{up}{1}
\walkpanel{3}{6}{125364}{0,2,1,2,0}{-90}{dn}{1}
\walkpanel{4}{6}{126354}{0,1,0,2,0}{+9}{up}{1}
\end{tikzpicture}
\caption{The $35$ late-growing permutations $\sigma\in\Phi_3\subset\mathbb{S}_6$, in the order of Table~\ref{tab:phi3}. 
Each panel shows the walk $k\mapsto T_k$, $0\leqslant k\leqslant 5$ (horizontal axis $k$, vertical axis $T_k$), and the contribution $(-)^\sigma\,c(\sigma,3)$ of that walk to the sum~$\consta(3)$.}\label{Fig35walks}
\end{figure}
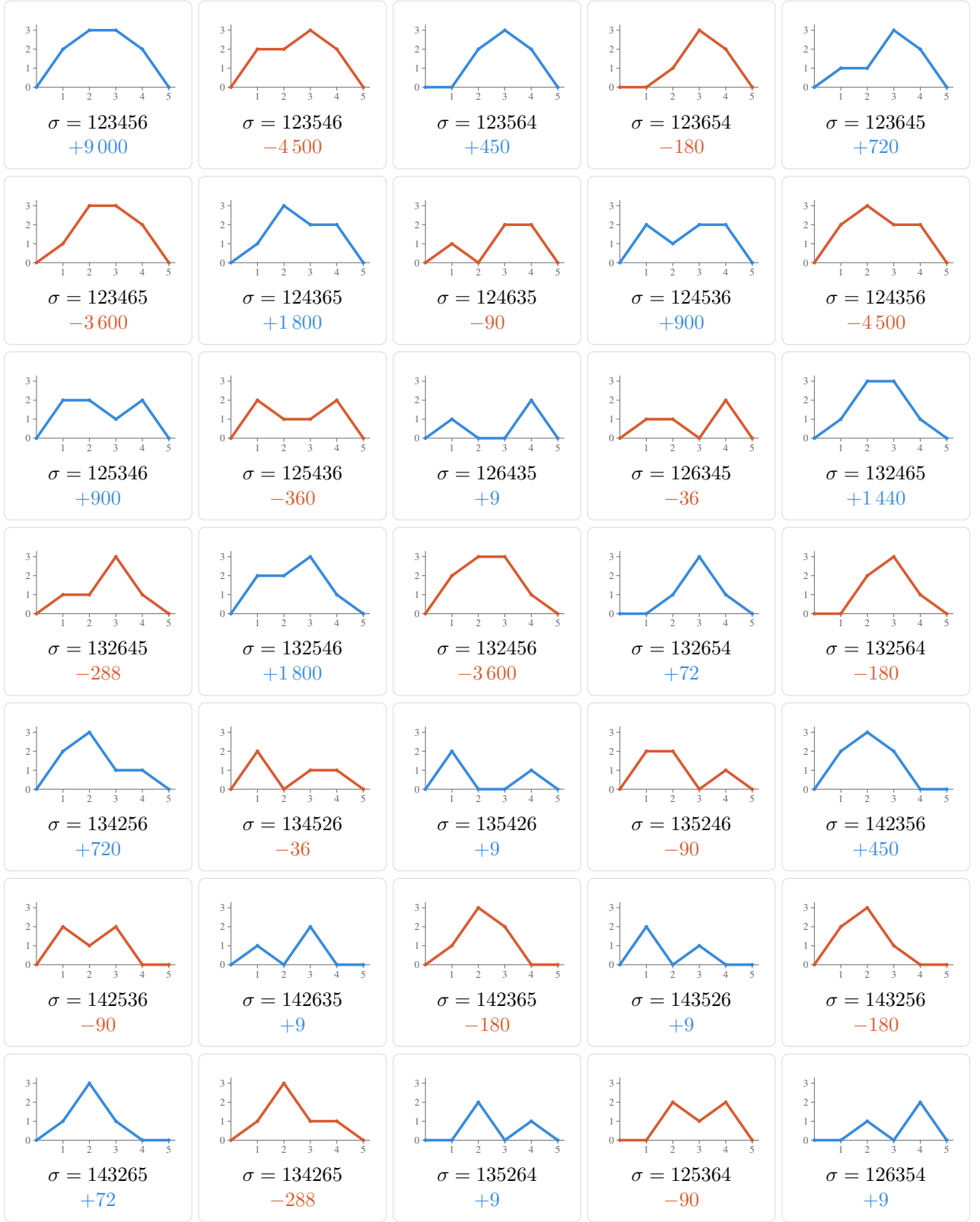


\begin{thebibliography}{17}\small

\bibitem{DzhumaWron}
\by{Dzhumadil'daev A.\,S.} (2005) $n$-Lie structures generated by Wronskians,
\jour{Siberian Math.~J.} \vol{46}:4, 601--612.\ 
(\jour{Preprint} \texttt{arXiv:math.RA/0202043})

\bibitem{DzhumaAssoc}
\by{Dzhumadil'daev A.\,S.} (1988) Integral and mod $p$-cohomologies of the Lie algebra $W_1$, \jour{Funct.\ Anal.\ Appl.} \vol{22}:3, 226--228.

\bibitem{HanlonWachs}
\by{Hanlon P., Wachs M.\,L.} (1995) On Lie $k$-algebras, \jour{Adv.\ 
Math.} \vol{113}, 206--236.

\bibitem{ForKac}
\by{Kiselev A. V.} (2005) 
On associative Schlessinger\/--\/Stasheff algebras and Wronskian determinants,
\jour{Fundam.\ Prikl.\ Mat.} \vol{11}:1, 159--180
(
English transl.:
\jour{J.~Math.\ Sci.} (2007) \vol{141}:1, 1016--1030).\ %
(\jour{Preprint} \texttt{arXiv:math.RA/0410185})

\bibitem{PRG25}
\by{Kiselev A.\,V.} (2025)
Wronskians as $N$-ary brackets in finite\/-\/dimensional analogues of~$\mathfrak{sl}(2)$,
\jour{J.~Phys.\textup{:}\ Conf.~Ser.} \vol{3152}, 
Paper 012044, 1--8.\ 
(\jour{Preprint} \texttt{arXiv:2510.02145} [math.RA])

\bibitem{Yer25}
\by{Kiselev A.\,V.} (2025) Jacobi identities for Wronskian determinants
over multidimension,
\jour{Bulgarian J.~Phys.} \vol{52}-s1  Quantum Theory \& Symmetries
, 102--107.\ %
(\jour{Preprint} \texttt{arXiv:2511.} \texttt{03848} [math.RA])

\bibitem{LadaStasheff}
\by{Stasheff J., Lada\ T.} (1993)
Introduction to SH~Lie algebras for physicists,
\jour{Internat.\ J.~Theoret.\ Phys.} \vol{32}:{7}, {1087--1103}.\ %
(\textit{Preprint} 
{\texttt{arXiv:hep-th/9209099}})

\bibitem{cpWron}
\by{Shah K.\,C., Kiselev A.\,V.} (2026)
The alternating compositions of weighted differential operators yield the weights' Wronskian with which constant\,? 
\jour{Preprint} \texttt{arXiv:2605.11137} [math.CO], 22~p.

\bibitem{Method}
\by{Shah K.\,C., Kiselev A.\,V.} (2026)
New integer sequence OEIS\:A392\,714 counts Wronskians: fast evaluation via late\/-\/growing permutations,
\jour{Preprint} \texttt{arXiv:}\texttt{2610.08636} [math.NT], 30~p.

\bibitem{Valladolid}
\by{Shah K.\,C., Kiselev A.\,V.} (2026)
New integer sequence OEIS\:A392\,714 counts Wronskians, II.
Computing within the permutation group algebra,
Proc.\ 36th Int.\ colloquium on group theoretic methods in Physics: \textsc{Group36} (13--17 July 2026, Valladolid, Spain),
\jour{Preprint} \texttt{arXiv:}\texttt{2611.\phantom{xxxxx}} [math.CO], 16~p.

\bibitem{StasheffAinftyLinfty}
\by{Stasheff J.} (2019)
L-infinity and A-infinity structures,
\jour{Higher Structures} \vol{3}:{1}, {292--326}.

\bibitem{oeisA147681}
The on-line encyclopedia of integer sequences: Sequence a147681, 
https://oeis.org/A147681, OEIS Foundation Inc., 2026.

\bibitem{oeisA392714}
The on-line encyclopedia of integer sequences: Sequence a392714, 
https://oeis.org/A392714, OEIS Foundation Inc., 2026.

\end{thebibliography}
\end{document}